\documentclass[10pt]{article} 

\usepackage[utf8]{inputenc} 
  
\usepackage{geometry} 
\usepackage{graphicx} 

\usepackage{booktabs} 
\usepackage{array} 
\usepackage{paralist} 
\usepackage{amsmath}
\usepackage{amssymb}
\usepackage{amsthm}
\usepackage{xifthen}
\usepackage{xcolor}
\usepackage{ulem}
\usepackage{hyperref}

\usepackage{subfigure}

\usepackage{soul}

\usepackage[commentmarkup=footnote]{changes}
\definechangesauthor[name={Andi}, color=red]{and}
\definechangesauthor[name={Kristoffer}, color=blue]{kri}
\definechangesauthor[name={Soeren}, color=purple]{soe}

\usepackage{fancyhdr} 
\usepackage{sectsty}
\allsectionsfont{\sffamily\mdseries\upshape} 

\newtheorem{theorem}{Theorem}
\newtheorem*{acknowledgement*}{Acknowledgement}

\newtheorem{assumption}{Assumption}

\newtheorem{corollary}[theorem]{Corollary}

\newtheorem{definition}[theorem]{Definition}

\newtheorem{lemma}[theorem]{Lemma}

\newtheorem{proposition}[theorem]{Proposition}
\newtheorem{remark}[theorem]{Remark}

\usepackage[nottoc,notlof,notlot]{tocbibind} 
\usepackage[titles,subfigure]{tocloft} 

\def\F{\mathcal F}

\DeclareMathOperator{\E}{\mathbb E}

\newcommand{\kom}[1]{}
\renewcommand{\kom}[1]{{\bf [#1]}}

\usepackage{accents}
\newcommand{\ubar}[1]{\underaccent{\bar}{#1}}
\newcounter{komcounter}
\numberwithin{komcounter}{section}

\title{Local time pushed mixed stopping and smooth fit for time-inconsistent stopping problems}
\author{
Andi Bodnariu\\Department of Mathematics, Stockholm University\\ 
\\ Sören Christensen\\Department of Mathematics, Kiel University\\
\\Kristoffer Lindensjö\\Department of Mathematics, Stockholm University
}

\begin{document}

\maketitle

\begin{abstract}  We consider the game-theoretic approach to {time-inconsistent} stopping of a one-dimensional diffusion where the {time-inconsistency} is due to the presence of a non-exponential (weighted) discount function. In particular, we study (weak) equilibria for this problem in a novel class of mixed (i.e., randomized) stopping times based on a local time construction of the stopping intensity. For a general formulation of the problem we provide a verification theorem giving sufficient conditions for mixed (and pure) equilibria in terms of a set of variational inequalities, including a smooth fit condition. We apply the theory to prove the existence of (mixed) equilibria in a recently studied real options problem in which no pure equilibria exist.
\end{abstract}




 







\section{Introduction}\label{intro}
In this paper we study the game-theoretic approach to {time-inconsistent} stopping where the {time-inconsistency} is due to the presence of a non-exponential discount function. In particular, we consider the problem of how to choose a stopping time $\tau$ that minimizes---in the (weak) equilibrium sense---the cost function 
\begin{align}\label{cost_equation}
J_{\tau}(x):=\E_x\left(\int_0^\tau h(s)f(X_s)ds+h(\tau)g(X_\tau)\right),
\end{align}
where 
(i)  $h:[0,\infty)\rightarrow (0,1]$ is weighted discount function (WDF), which means that it is strictly decreasing and allows the representation 
\begin{align*}
h(t)=\int_0^\infty e^{-rt}dF(r),
\end{align*} 
where $F$ a distribution function concentrated on $[0,\infty)$, 
(ii) the functions $f$ and $g$ are non-negative, and 
(iii) the process $X=\left(X_t\right)_{t\geq0}$ is the strong solution to the one-dimensional stochastic differential equation (SDE)
\begin{align} \label{the-diffusion}
dX_t = \mu(X_t)dt + \sigma(X_t)dW_t, \enskip X_0 = x.
\end{align}
Section \ref{sec:model} gives more details on the mathematical model and the problem while Section \ref{sec:prevlit} gives a background on the game-theoretic approach to time-inconsistent problems. We study equilibria for this problem in a novel class of mixed (i.e., randomized) stopping times and the investigation of this class of stopping times is one of the main contributions of this paper.  

The game-theoretic approach to stopping problems with time-inconsistency due to non-exponential discounting was first considered in \cite{huang2018time}. Our investigations are inspired by a recent paper by Tan, Wei and Zhou; see \cite{tan2021failure}, where a weak equilibrium approach is used---we refer to \cite{bayraktar2022equilibria} for an overview and a comparison of different equilibrium concepts for time-inconsistent stopping problems. The main difference between \cite{tan2021failure} and the present paper is that in \cite{tan2021failure} only pure stopping strategies (essentially corresponding to exit times from sets in the state space of $X$) are admissible whereas we study the problem when allowing for the novel mixed stopping strategies mentioned above. The importance of our contribution is explained in the next paragraphs.

In \cite{tan2021failure}, a verification theorem giving sufficient conditions for pure equilibria in terms of a differential equation is provided. The main structural observation of \cite{tan2021failure} is obtained by studying  a specification of the problem   
corresponding to the discount function
\begin{align*}
h(t)=pe^{-r_1t}+(1-p)e^{-r_2t},
\end{align*} 
where $p\in(0,1)$ and $r_2>r_1>0$, and the cost function
\begin{align}\label{cost-func-real-op}
J_{\tau}(x)=\E_x\left(\int_0^\tau h(s)X_sds+h(\tau)K\right),
\end{align}
assuming that $X$ is a geometric Brownian motion (GBM) and that $K>0$ is a constant; see \cite[Section~3.3]{tan2021failure}, where this problem is also motivated as a real options problem. For this specification of the problem it is in \cite{tan2021failure} shown that a smooth fit condition is necessary in order for a pure threshold stopping strategy to be an equilibrium (see \cite{bayraktar2022equilibria} for a general result in this direction) but it is also shown that obtaining a reasonable candidate solution fulfilling smooth fit is not guaranteed to result in an equilibrium. 
In particular, it is found that 
(i) if a certain condition for the model parameters is satisfied then a threshold stopping time can be identified as an equilibrium using a smooth fit principle, and 
(ii) if this condition is not satisfied, however, then there is still a threshold time with a smooth value function, but this does not result in an equilibrium, and, moreover, there exists is in fact no (pure) equilibrium in this case.

To facilitate the existence of equilibria in the weak equilibrium approach to time-inconsistent stopping in a one-dimensional SDE setting, a class of mixed strategies is introduced in \cite{christensen2020time}, but there the time-inconsistency is due to the consideration of a non-linear function of an expected reward. 
The notion of mixed strategies in \cite{christensen2020time} corresponds to stopping according to a state-dependent intensity of Lebesgue density type. 
This class is, however, not suitable to deal with the problem \eqref{cost_equation}. The notion of mixed stopping strategy in the present paper therefore  generalizes that of \cite{christensen2020time} by allowing, loosely speaking, the accumulated intensity to increase in a singular fashion using a local time construction; see Definition \ref{def:mix_strat} and Remark \ref{rem:motivation_defs}. We call these strategies  local time pushed mixed stopping strategies. In relation to this we mention \cite{ekstrom2017dynkin} in which a two-player Dynkin game with heterogeneous beliefs is studied and where, in \cite[Example 5.4]{ekstrom2017dynkin}, a randomized stopping time based on a similar local time construction is considered.

These are the main results of the present paper: (i) for the general problem presented above we formulate and prove a verification theorem giving sufficient conditions for mixed (and pure) equilibria in terms of a set of variational inequalities, which includes a smooth fit condition, and (ii) for the real options specification of the problem \eqref{cost-func-real-op}, we show that an equilibrium always exists when allowing for local time pushed mixed stopping and find this equilibrium using an ansatz based on the smooth fit condition.

Further previous literature is discussed in Section \ref{sec:prevlit}. In Section \ref{sec:model} we give the definitions of pure and mixed stopping strategies and of the equilibrium, as well as further details of the underlying mathematical model.  In Section \ref{sec:ver-res} we formulate and prove the equilibrium verification theorem. In Section \ref{sec:mixed-th-strat} we establish further results for mixed stopping strategies, which are later used in an ansatz to finding the equilibrium solution for the real options problem. The real options problem is studied in Section \ref{sec:real-option-problem}.

%
%
%
%
%
%
%

\subsection{Previous literature}\label{sec:prevlit}
The underlying decision-theoretic ideas for the problem considered here go back to the concept of {consistent planning} by Strotz \cite{strotz}. For the mathematical literature, we first mention the recently published monograph \cite{bjork2021time} which contains an overview of the game-theoretic approach to time-inconsistent stochastic control and stopping problems. For a general motivation and background to the game-theoretic approach to time-inconsistent stopping problems, and in particular the weak equilibrium formulation, we refer to 
\cite{christensen2018finding} where the time-inconsistency is due to the consideration of an expected value of a reward function which depends on the initial state and pure stopping is considered in both a general Markov process setting and a diffusion setting, as well as \cite{christensen2020time} and \cite{tan2021failure} (both mentioned in Section \ref{intro}). The weak equilibrium approach is also considered in \cite{liang2021weak} where a time-inconsistent combined control and stopping problem is studied for a general diffusion and a differential equation type verification theorem is provided for pure stopping strategies. 

In the literature for time-inconsistent stopping problems there are three complementing types of equilibrium definitions: weak, mild and strong, and in this paper we always consider weak equilibria. The relationship between these equilibrium definitions is investigated in \cite{bayraktar2022equilibria} in a setting corresponding to a one-dimensional diffusion, pure stopping strategies and time-inconsistency due to the presence of a non-exponential discount function satisfying what is known as the log sub-additivity condition. In particular, it is found that 
an optimal mild equilibrium is a weak equilibrium and conditions implying that a weak equilibrium is also a strong equilibrium are found. A variational inequality characterization of weak equilibria is provided. A similar analysis is performed in \cite{https://doi.org/10.1111/mafi.12293}, but there the state process is a continuous time Markov chain.

The mild equilibrium definition is studied in  \cite{huang2018time} in a setting corresponding to 
a non-exponential discount function satisfying log sub-additivity,
a one-dimensional diffusion, and pure stopping strategies.  In \cite{huang2021optimal} the multi-dimensional diffusion version of this problem is studied and an existence result for  optimal mild equilibria is provided. A discrete time version of the problem is studied in \cite{huang2019optimal} and existence of an optimal mild equilibrium is established.  In \cite{huang2022time} a two-player Dynkin game under non-exponential discounting in discrete time is studied. 

Mixed stopping strategies in the context of time-inconsistent stopping problems for the discrete time case are studied in \cite{bayraktar2019time} and \cite{christensen2020Timemyopic}.  For related results in an MDP setting, we refer to \cite{Jaskiewicz2021} and the references therein. Mixed stopping strategies for continuous time two-player stopping games are studied in, e.g., \cite{de2020value,ekstrom2017dynkin,riedel2017subgame,touzi2002continuous} and for a continuous time controller-stopper game with unknown competition in \cite{ekstrom2022detect}.  Note, however, that in these games, there are two players choosing randomized stopping times as their strategies while in our framework we have a continuum of players (one for each state $x$) and the randomized stopping time is their strategy profile.

\section{Model and problem formulation}\label{sec:model}
All random quantities live on a filtered probability space $(\Omega, \F, \left({\cal F}_t\right)_{t\geq 0}, \mathbb{P}_x)$ satisfying the usual conditions and carrying the Wiener process $W=\left(W_t\right)_{t\geq0}$. 
We denote the associated expectations and probabilities by $\mathbb{E}_x$ and $\mathbb{P}_x$, respectively. Further assumptions are collected here.

\begin{assumption} 
The state process $X=(X_t)_{t \geq 0}$
lives on an open interval $(\alpha,\beta)$ with $-\infty\leq \alpha \leq \beta  \leq  \infty$, 
satisfies the strong Markov property, 
and is the strong solution to the SDE \eqref{the-diffusion} whose coefficients 
$\mu:(\alpha,\beta)\rightarrow  \mathbb{R}$ 
and 
$\sigma:(\alpha,\beta) \rightarrow  (0,\infty)$ 
are continuous. The function $x\mapsto \mathbb{P}_x (A)$ is  measurable for any $A \in {\cal F}$. 
There exists a measurable time shift operator $\theta$ such that 
$X_\tau \circ \theta_{\tau_h}  =   X_{\tau \circ \theta_{\tau_h} + \tau_h}$ for any stopping time $\tau$ and  
$\tau_h:=\inf\{t\geq 0: |X_t-X_0| \geq h\}$. 
The function $f:(\alpha,\beta)\rightarrow [0,\infty)$ in \eqref{cost_equation} is continuous whereas 
$g:(\alpha,\beta)\rightarrow (0,\infty)$ is twice continuously differentiable and bounded. Lastly, we use the notation
\begin{align}\label{w_eq}
w_{\tau}(x,r):=\E_x\left(\int_0^\tau e^{-rs}f(X_s)ds+e^{-r\tau}g(X_\tau)\right)
\end{align}
and assume that for some $C(x,r)>0$ it holds that $\sup_\tau w_{\tau}(x,r) \leq C(x,r)$ where the supremum is taken over the set of stopping times $\mathcal N$ (defined below) and $x\mapsto \int_0^\infty C(x,r)dF(r)+\int_0^\infty rC(x,r)dF(r)$ is locally bounded. It is also assumed that $F$ has a finite second moment.  
%
%
\end{assumption}
Note that \eqref{w_eq} allows, as in \cite{tan2021failure}, the cost function \eqref{cost_equation} to be written as
\begin{align}\label{J-in-terms-w}
J_{\tau}(x)=\int_0^\infty w_\tau(x,r)dF(r).
\end{align}
Throughout the paper we use use the convention that
\begin{align*}
e^{-r \tau}g(X_{\tau}) = 0 \enskip \mbox{on} \enskip \{\tau=\infty\}.
\end{align*}
The augmented filtration generated by $X$ is denoted by $\left({\cal F}^X_t\right)_{t \geq 0}$. For the differential operator associated to the SDE \eqref{the-diffusion} we use the notation
\begin{align*}
\mathbf{A} =\mu(x)\frac{d }{d x}+\frac{1}{2}\sigma^2(x)\frac{d^2 }{d x^2}.
\end{align*}
In the following we define the notions of pure and mixed stopping times and of the (Nash) equilibrium; see the references in Section \ref{sec:prevlit}---in particular \cite{christensen2018finding,christensen2020time}---as well as Remark \ref{rem:motivation_defs} (below) for interpretations.

\begin{definition}[Pure and  mixed stopping times] \label{def:mix_strat}  Let $D \subset (\alpha,\beta)$ be an open set and let
\begin{align*}
\tau^D :=  \inf\{t\geq 0: X_t \notin D\}
\end{align*}
which is said to be a pure (strategy) stopping time. Let  
\begin{align*}
\tau^{\Psi,\mathcal{X}}:=\inf\{t\geq 0:\Psi_t\geq U\}
\end{align*}
where $U\sim exp(1)$ is an exponentially distributed random variable that is independent of all other random sources and
\begin{align}\label{Psi-def}
\Psi_t:=\sum_{i=1}^{n}\psi_i l_t^{x_i}+\int_0^t \psi(X_s)ds,
\end{align}
where $n$ is a fixed number, 
$x_i\in\mathcal{X}:=\{x_1,\dots,x_n\} \subset D$, 
$\psi_i\in (0,\infty)$ are fixed constants,
$\psi:D \rightarrow [0,\infty)$ is a right continuous with left real limits (RCLL) function with finitely many discontinuities (on the complement $D^c$ we define the function $\psi(\cdot)$ to be identically equal to zero), 
and $l^{x_i}=\left(l_t^{x_i}\right)_{t\geq 0}$ is the local time of $X$ at $x_i$, 
\begin{align*}
l_t^{x_i}:=|X_t-{x_i}|-|X_0-{x_i}|-\int_0^t \mbox{sgn}(X_s-x)dX_s.
\end{align*}
Lastly, let
\begin{align*}
\tau^{\Psi,\mathcal{X},D} :=  \tau^{\Psi,\mathcal{X}}\wedge \tau^D
\end{align*}
which is said to be a local time pushed mixed (strategy) stopping time.  If the function $x \mapsto w_{\tau^{\Psi,\mathcal{X},D}}(x,r)$, cf. \eqref{w_eq}, is continuous for any fixed $r$ (in the support of $F$), then $\tau^{\Psi,\mathcal{X},D}$ is said to be admissible and we write $\tau^{\Psi,\mathcal{X},D} \in \mathcal N$ (which denotes the set of admissible stopping times).
\end{definition}

\begin{remark}\label{rem:motivation_defs}
In \eqref{Psi-def} the function $\psi(\cdot)$  corresponds to a state-dependent stopping intensity of Lebesgue density type, whereas the constants $\psi_i$ correspond to state dependent singular increases in the accumulated stopping intensity induced by local times. We also remark that the intensity funcion 
$\psi(x)$ is allowed to explode as $x \in D$ is sent to $\partial D$. The function $\lambda(\cdot)$ and the constant $\lambda$ in Definition \ref{def:mixed_threshold_stopping_time} (below) have analogous interpretations. 
\end{remark}

When there is no risk of confusion we will sometimes ease notation by writing 
$\tau^{\Psi}$ instead of $\tau^{\Psi,\mathcal{X}}$ and
$\tau^{\Psi,D}$ instead of $\tau^{\Psi,\mathcal{X},D}$. For any two stopping times $\tau^{(1)}$ and $\tau^{(2)}$ we let
\begin{align*} 
\tau^{(1)} \diamond \tau^{(2)}(h):=
I_{\{\tau^{(2)} \leq \tau_h\}}\tau^{(2)} + I_{\{\tau^{(2)} > \tau_h\}}(\tau^{(1)}   \circ \theta_{\tau_h}+\tau_h).
\end{align*}
Now we formalize the idea that a stopping time $\hat\tau  \in\mathcal{N}$ represents an equilibrium if deviations from it do not provide an advantage (to first order) in any state.

\begin{definition} [Equilibrium]\label{def:equ_stop_time}
A stopping time $\hat\tau  \in\mathcal{N}$ is said to be a (weak) equilibrium stopping time if
\begin{align}
\limsup_{h\searrow 0}\frac{J_{\hat\tau}(x)-J_{\hat\tau\diamond \tau^{\Psi,\mathcal{X},D}(h)}(x)}{\E_x(\tau_h)}&\leq 0\label{eqdef2}
\end{align}
for each $\tau^{\Psi,\mathcal{X},D}\in\mathcal{N}$ and $x\in (\alpha,\beta)$. 
\end{definition}
A purpose of the present paper is to find equilibrium stopping times of threshold type which we define as follows. 

\begin{definition} [Mixed and pure threshold stopping times] \label{def:mixed_threshold_stopping_time}
Let 
\begin{align*}
\tau^{\Lambda,\ubar{x},\bar{x}}:=\tau^\Lambda \wedge \inf\{t\geq 0:X_t\geq \bar{x}\}
\end{align*}
where
\begin{align*}
\tau^\Lambda:=\inf\{t\geq 0:\Lambda_t\geq U\}
\end{align*}
and
\begin{align*}
\Lambda_t:=\lambda l_t^{\ubar{x}}+\int_0^t\lambda(X_t)dt,
\end{align*}
where $\lambda>0$ is a constant, 
$\alpha \leq \ubar{x}<\bar{x}<\beta$ and $\lambda:(\alpha,\bar{x})\rightarrow [0,\infty)$ is a RCLL function (with real left limits) and finitely many discontinuities, satisfying 
\begin{align*}
\begin{split}
&\lambda(x)=0 \enskip \mbox{ for $x  \in (\alpha,\ubar{x})$},\\
&\lambda(x)\geq 0 \enskip \mbox{ for $x  \in [\ubar{x},\bar{x})$}. 
\end{split}
\end{align*}
(On $[\bar{x},\infty)$ we define $\lambda(\cdot)$ to be identically equal to zero). The stopping time $\tau^{\Lambda,\ubar{x},\bar{x}}$ is said to be a mixed threshold stopping time. For $\ubar{x}=\bar{x}$, we let $\tau^{\Lambda,\ubar{x},\bar{x}}:=\inf\{t\geq 0:X_t\geq \ubar{x}\}$ and call this a pure threshold stopping time. \end{definition}


\subsection{Initial observations}
The analysis in the subsequent sections relies on the results in this section. In this section we write $\tau^\Psi$ instead of $\tau^{\Psi,\cal X}$ for ease of exposition.  

\begin{proposition}  \label{dist-for-Lambda:prop}
The distribution of $\tau^\Psi$ conditioned on observing the whole path of $X$ is given by 
\begin{align*}
\mathbb{P}_x\left( \tau^\Psi  \leq t \enskip \vline \enskip {\cal F} ^X_\infty  \right) 
=  1 - e^{-\Psi_\tau} - \int_t^\tau e^{-\Psi_s} d\Psi_s,
\end{align*}
where $\tau:= \inf\{t\geq 0: \Psi_t=\infty\}$ 
(we use the convention $\inf \emptyset = \infty$ and that the integral from $s$ to $t$ is zero if  $\Psi_t=\Psi_s=\infty$). 
Moreover, 
$$\mathbb{P}_x\left( \tau^\Psi = \infty \enskip \vline \enskip {\cal F}^X_\infty \right) = e^{-\Psi_\tau}.$$
\end{proposition}

\begin{proof}

Using first that $U \sim exp(1)$ and second that $\Psi$ is an increasing process and It\^{o}'s formula for semimartingales we find 
\begin{align*}
\mathbb{P}_x\left( \tau^\Psi > t \enskip \vline \enskip {\cal F} ^X_\infty \right) 
& =  \mathbb{P}_x\left(  \Psi_t < U \enskip \vline \enskip {\cal F} ^X_\infty \right)\\
& = e^{-\Psi_t}\\
& = e^{-\Psi_s} - \int_s^t e^{-\Psi_u} d\Psi_u,
\end{align*}
for any $s\geq0$ and the first result follows.   
To see that the second result holds observe that 
\begin{align*}
e^{-\Psi_\tau} & =\lim_{t\rightarrow \infty}\mathbb{P}_x\left( \tau^\Psi > t \enskip \vline \enskip {\cal F} ^X_\infty \right)\\
& = \mathbb{P}_x\left( \tau^\Psi = \infty \enskip \vline \enskip {\cal F} ^X_\infty\right).
\end{align*}
\end{proof}

\begin{proposition} \label{centralProposition} For any fixed constant $r\geq0$ and $x \in D$, it holds that
\begin{align}
\E_x\left(\int_0^{\tau^{\Psi,D}}e^{-rt} g(X_t)d\Psi_t  \right) 
=
\E_x\left( e^{-r \tau^{\Psi}}g(X_{\tau^{\Psi}}) I_{\{\tau^\Psi \leq \tau^D\}}  \right).\label{eq1:centralProposition}
\end{align}
\end{proposition}

\begin{proof} Using Proposition \ref{dist-for-Lambda:prop} we rewrite the left hand side in \eqref{eq1:centralProposition} as follows (relying on, e.g., Fubini, the tower property and $e^{-\Psi_t} = \int_0^\infty I_{\{\Psi_t \leq u\}} e^{-u}du $) 
%
%
%
%
\begin{align*}
\E_x\left( \int_0^\infty e^{-r t}g(X_t) I_{\{t \leq \tau^D\}} I_{\{\Psi_t<U\}} d\Psi_t \right)
& = \E_x\left(\E_x\left( \int_0^\infty e^{-r t}g(X_t) I_{\{t \leq \tau^D\}} I_{\{\Psi_t<U\}} d\Psi_t \enskip \vline \enskip {\cal F} ^X_\infty \right)  \right)\\
& = \E_x\left(\int_0^\infty e^{-u}\int_0^\infty e^{-r t}g(X_t) I_{\{t \leq \tau^D\}} I_{\{\Psi_t<u\}} d\Psi_t du  \right)\\
& = \E_x\left(\int_0^\infty  e^{-r t}g(X_t) I_{\{t \leq \tau^D\}} \int_0^\infty e^{-u} I_{\{\Psi_t<u\}} du d\Psi_t  \right)\\
& = \E_x\left( \int_0^\infty e^{-r t}g(X_t) I_{\{t\leq \tau^D\}}e^{-\Psi_t} d\Psi_t\right)\\
& = \E_x\left( \int_0^\infty e^{-r t}g(X_t) I_{\{t \leq \tau^D\}}  
d\mathbb{P}_x\left( \tau^\Psi  \leq t \enskip \vline \enskip {\cal F} ^X_\infty  \right) \right)\\
&=\E_x\left(\E_x\left( e^{-r \tau^{\Psi}}g(X_{\tau^{\Psi}}) I_{\{\tau^\Psi \leq \tau^D\}} I_{\{\tau^\Psi <\infty\}} \enskip \vline \enskip {\cal F} ^X_\infty \right) \right).
\end{align*}
The result follows.
\end{proof}

\section{Main theoretical results}\label{sec:main}
\subsection{A verification theorem}\label{sec:ver-res}

The following result presents sufficient conditions for equilibria of the threshold type (either pure or mixed). 

\begin{theorem}[Verification]\label{main_thm}
Let $\tau^{\Lambda,\ubar{x},\bar{x}}\in \cal {N}$ be a mixed or pure threshold stopping time. If the corresponding cost function, cf. \eqref{cost_equation}, satisfies $J_{\tau^{\Lambda,\ubar{x},\bar{x}}}\in \mathcal{C}^2((\alpha,\ubar{x})\cup(\ubar{x},\bar{x})) \cap \mathcal{C}(\alpha,\beta)$ and the conditions 
\begin{align}
J_{\tau^{\Lambda,\ubar{x},\bar{x}}}(x) - g(x)&< 0, \enskip \mbox{ for } x \in (\alpha,\ubar{x})\label{mainthmcond1}  \tag{I}\\
J_{\tau^{\Lambda,\ubar{x},\bar{x}}}(x) - g(x)&=0, \enskip \mbox{ for } x \in [\ubar{x},\bar{x}) \label{mainthmcond2}  \tag{II}\\
f(x)+\mathbf{A}g(x)-g(x)\int_0^\infty rdF(r)  &\geq 0, \enskip \mbox{ for } x\in [\bar{x},\beta)\label{mainthmcond3} \tag{III}\\
J'_{\tau^{\Lambda,\ubar{x},\bar{x}}}(\ubar{x})&= g'(\ubar{x}), \enskip \mbox{(smooth fit)} \label{mainthmcond4} \tag{IV}
\end{align}
are satisfied, then $\tau^{\Lambda,\ubar{x},\bar{x}}$ is an equilibrium stopping time. 
\end{theorem}
\begin{remark}  Theorem \ref{main_thm} can be generalized to regard local time pushed mixed equilibrium stopping times (Definition \ref{def:mix_strat}) of more complicated structure than those of mixed and pure threshold stopping times $\tau^{\Lambda,\ubar{x},\bar{x}}$ (Definition \ref{def:mixed_threshold_stopping_time}) using the techniques of the present paper. We have chosen the current formulation for ease of exposition.
\end{remark}
 
The remainder of this subsection will be devoted to proving Theorem \ref{main_thm}. The proof approach relies largely on rewriting the expression in the left hand side of the equilibrium condition \eqref{eqdef2}, when considering $\hat \tau = \tau^{\Lambda,\ubar{x},\bar{x}}$, according to  
\begin{align}
\label{split-up-NE-cond}
\begin{split}
&\frac{J_{\tau^{\Lambda,\ubar{x},\bar{x}}}(x)-J_{\tau^{\Lambda,\ubar{x},\bar{x}}\diamond \tau^{\Psi,\mathcal{X},D}(h)}(x)}{\E_x(\tau_h)}\\
&\enskip =\frac{J_{\tau^{\Lambda,\ubar{x},\bar{x}}}(x)-J_{\tau^{\Lambda,\ubar{x},\bar{x}}\circ \theta_{\tau_h}+\tau_h}(x)}{\E_x(\tau_h)}+\frac{J_{\tau^{\Lambda,\ubar{x},\bar{x}}\circ \theta_{\tau_h}+\tau_h}(x)-J_{\tau^{\Lambda,\ubar{x},\bar{x}}\diamond \tau^{\Psi,\mathcal{X},D}(h)}(x)}{\E_x(\tau_h)}
\end{split}
\end{align}
and then studying what happens to these terms separately when sending $h\searrow 0$ for different values of $x$, in the case $x \in D$. Relying on these results we show that conditions \eqref{mainthmcond1}--\eqref{mainthmcond4} are sufficient for $\tau^{\Lambda,\ubar{x},\bar{x}}$ to be an equilibrium at the end of this subsection. The case $x\notin D$ turns out to be trivial.

When considering stopping strategies with an immediate local time push 
we will see that the left hand side of the equilibrium condition \eqref{eqdef2} may diverge. 
In order to investigate the nature of this convergence (cf. Proposition \ref{prop_CD_local} below) we need the following result, where the denominator has been replaced with the larger term $\sqrt{\E_x(\tau_h)}$ in order to ensure convergence.

\begin{lemma} \label{main_lemma_local}  For any $\tau^{\Lambda,\ubar{x},\bar{x}},\tau^{\Psi,\mathcal{X},D} \in \mathcal{N}$ and $x = x_i\in \mathcal{X}$, it holds that
\begin{align*}
\lim_{h\searrow 0}&\frac{J_{\tau^{\Lambda,\ubar{x},\bar{x}}\circ \theta_{\tau_h}+\tau_h}(x)-J_{\tau^{\Lambda,\ubar{x},\bar{x}} \diamond \tau^{\Psi,\mathcal{X},D}(h)}(x)}{\sqrt{\E_x(\tau_h)}}
=\psi_i\sigma(x)(J_{\tau^{\Lambda,\ubar{x},\bar{x}}}(x)-g(x)).
\end{align*}
Moreover, for any $\tau^{\Lambda,\ubar{x},\bar{x}},\tau^{\Psi,\mathcal{X},D} \in \mathcal{N}$ and $x\in (\alpha,\bar{x})\cap D$,
\begin{align*}
\lim_{h\searrow 0}\frac{J_{\tau^{\Lambda,\ubar{x},\bar{x}}}(x)-J_{\tau^{\Lambda,\ubar{x},\bar{x}}\diamond \tau^{\Psi,\mathcal{X},D}(h)}(x)}{\sqrt{\E_x(\tau_h)}} = \sigma(x)\left(I_{\{x=\ubar{x}\}}\lambda-\sum_{i=1}^nI_{\{x=x_i\}}\psi_i\right)\left(g(x)-J_{\tau^{\Lambda,\ubar{x},\bar{x}}}(x)\right).
\end{align*}
\end{lemma}
\begin{proof} The set $D$ is open by definition. Hence, for every $x\in D$, there exists a $\bar{h}$ such that $[x-\bar{h},x+\bar{h}]\subset D$. Hence, we find for $0<h\leq\bar{h}$ that $\tau_h< \tau^D$ and
\begin{align*}
\tau^{\Lambda,\ubar{x},\bar{x}} \diamond \tau^{\Psi,\mathcal{X},D}(h) 
& = I_{\{\tau^{\Psi,\mathcal{X},D} \leq \tau_h\}}\tau^{\Psi,\mathcal{X},D} + I_{\{\tau^{\Psi,\mathcal{X},D} > \tau_h\}}({\tau^{\Lambda,\ubar{x},\bar{x}}}\circ \theta_{\tau_h}+\tau_h)\\
& = I_{\{\tau^{\Psi,\mathcal{X}} \leq \tau_h\}}\tau^{\Psi,\mathcal{X}} + I_{\{\tau^{\Psi,\mathcal{X}} > \tau_h\}}({\tau^{\Lambda,\ubar{x},\bar{x}}}\circ \theta_{\tau_h}+\tau_h).
\end{align*} 
For the rest of the proof we write $\tau^\Psi$ instead of $\tau^{\Psi,\mathcal{X}}$. Consider $x=x_i$ for some $i$. Using the previous observation, Lemma \ref{mp_lemma} (in the appendix), the strong Markov property, conditioning on the whole path of $X$, and arguments similar to those of the proof of Proposition \ref{centralProposition}, we obtain
\begin{align*}
&J_{\tau^{\Lambda,\ubar{x},\bar{x}}\circ \theta_{\tau_h}+\tau_h}(x)-J_{\tau^{\Lambda,\ubar{x},\bar{x}} \diamond \tau^{\Psi,\mathcal{X},D}(h)}(x)\\
%
%
&=\int_0^\infty\E_x\left(I_{\{\tau^\Psi\leq \tau_h\}}\left(e^{-r\tau_h}w_{\tau^{\Lambda,\ubar{x},\bar{x}}}(X_{\tau_h},r)+\int_{\tau^\Psi}^{\tau_h} e^{-rs}f(X_s)ds-e^{-r\tau^\Psi}g(X_{\tau^\Psi})
\right)\right)dF(r)\\
\enskip 
&=\int_0^\infty\E_x\left(\int_0^{\tau_h}\left(e^{-r\tau_h}w_{\tau^{\Lambda,\ubar{x},\bar{x}}}(X_{\tau_h},r)+\int_{t}^{\tau_h} e^{-rs}f(X_s)ds-e^{-rt}g(X_t)
\right)d\mathbb{P}_x\left( \tau^\Psi  \leq t \enskip \vline \enskip {\cal F} ^X_\infty  \right)\right)dF(r)\\
\enskip &=\int_0^\infty\E_x\left(\int_0^{\tau^\Psi\wedge\tau_h}\left(e^{-r\tau_h}w_{\tau^{\Lambda,\ubar{x},\bar{x}}}(X_{\tau_h},r)+\int_{t}^{\tau_h} e^{-rs}f(X_s)ds-e^{-rt}g(X_t)
\right)d\Psi_t\right)dF(r)\\
&\enskip= 
\int_0^\infty\E_x\left(\int_0^{\tau^\Psi\wedge\tau_h}\psi(X_t)\left(e^{-r\tau_h}w_{\tau^{\Lambda,\ubar{x},\bar{x}}}(X_{\tau_h},r)+\int_{t}^{\tau_h} e^{-rs}f(X_s)ds-e^{-rt}g(X_t)\right)dt\right)dF(r)\\
&\quad+\int_0^\infty\E_x\left(\int_0^{\tau^\Psi\wedge\tau_h}\psi_i\left(e^{-r\tau_h}w_{\tau^{\Lambda,\ubar{x},\bar{x}}}(X_{\tau_h},r)+\int_{t}^{\tau_h} e^{-rs}f(X_s)ds-e^{-rt}g(X_t)
\right)dl_t^x\right)dF(r)\\
&\enskip=(A)+(B),
\end{align*}
where $\psi_i$ is the local push factor corresponding to $x$. 
Note that 
\begin{align*}
(A) 
& \leq \int_0^\infty\E_x\left(\int_0^{\tau^\Psi\wedge\tau_h}\psi(X_t)\left(w_{\tau^{\Lambda,\ubar{x},\bar{x}}}(X_{\tau_h},r)+\int_0^{\tau_h} f(X_s)ds  + g(X_t)\right)dt\right)dF(r).
\end{align*}
Hence, using also that the integrands are bounded (on the stochastic interval $[0,\tau_h]$), Lemma \ref{lemma-tau-h-oct} and \eqref{J-in-terms-w}, we obtain 
\begin{align*}
\frac{|(A)|}{\sqrt{\E_x(\tau_h)}}\leq \sup_{y_1,y_2,y_3,y_4\in [x-h,x+h]}\psi(y_1)\left(\left(J_{\tau^{\Lambda,\ubar{x},\bar{x}}}(y_2)+g(y_3)\right)
\frac{\E_x(\tau_h)}{\sqrt{\E_x(\tau_h)}}+f(y_4)\frac{\E_x(\tau_h^2)}{\sqrt{\E_x(\tau_h)}}\right)\rightarrow 0,
\end{align*}
as $h \searrow 0$.
%
%
%
In $(B)$, we see using similar arguments and Lemma \ref{prop_limit_local_time_time_tau-h} that the middle term satisfies
\begin{align*}
\lim_{h\searrow 0}\frac{\left| \int_0^\infty \E_x\left( \int_0^{\tau^\Psi\wedge\tau_h}\left(\psi_i\int_{t}^{\tau_h} e^{-rs}f(X_s)ds\right)dl_t^x\right) dF(r) \right|}{\sqrt{\E_x(\tau_h)}}
& \leq \lim_{h\searrow 0}\sup_{y\in[x-h,x+h]}\frac{\psi_if(y)\E_x\left(\tau_hl_{\tau_h}^x\right)}{\sqrt{\E_x(\tau_h)}}\\
& =0.
\end{align*}
For the first term in (B), we write  
\begin{align*}
& \int_0^\infty  \E_x\left(\int_0^{\tau^\Psi\wedge\tau_h}\psi_ie^{-r\tau_h}w_{\tau^{\Lambda,\ubar{x},\bar{x}}}(X_{\tau_h},r)dl_t^x\right)dF(r)\\
&\enskip= \int_0^\infty \E_x\left(\int_0^{\tau^\Psi\wedge\tau_h}\psi_i(e^{-r\tau_h}-1)w_{\tau^{\Lambda,\ubar{x},\bar{x}}}(X_{\tau_h},r)dl_t^x\right) dF(r)
+\int_0^\infty \E_x\left(\int_0^{\tau^\Psi\wedge\tau_h}\psi_iw_{\tau^{\Lambda,\ubar{x},\bar{x}}}(X_{\tau_h},r)dl_t^x\right)dF(r)\\
&\enskip= \int_0^\infty \E_x\left(\int_0^{\tau^\Psi\wedge\tau_h}\psi_i(e^{-r\tau_h}-1)w_{\tau^{\Lambda,\ubar{x},\bar{x}}}(X_{\tau_h},r)dl_t^x\right) dF(r)
+ \E_x\left(\int_0^{\tau^\Psi\wedge\tau_h}\psi_iJ_{\tau^{\Lambda,\ubar{x},\bar{x}}}(X_{\tau_h})dl_t^x\right)\\
&\enskip=(B1)+(B2).
\end{align*}
We now use $|e^{-y}-1|\leq y$ for $y\geq 0$, to obtain 
\begin{align*}
 \lim_{h\searrow 0}\frac{|(B1)|}{\sqrt{\E_x(\tau_h)}}\leq \lim_{h\searrow 0}\sup_{y\in[x-h,x+h]}
\psi_i \int_0^\infty r w_{\tau^{\Lambda,\ubar{x},\bar{x}}}(y,r)dF(r)\frac{\E_x\left(\tau_hl_{\tau_h}^x\right)}{\sqrt{\E_x(\tau_h)}}=0,
 \end{align*}
by Lemma \ref{prop_limit_local_time_time_tau-h}. For $(B2)$ we obtain, again using Lemma \ref{prop_limit_local_time_time_tau-h}, 
\begin{align*}
&\left|\frac{(B2)}{\sqrt{\E_x(\tau_h)}}-\psi_i\sigma(x)J_{\tau^{\Lambda,\ubar{x},\bar{x}}}(x)\right|\\
&\enskip = 
\left|\psi_i\frac{\E_x\left(\int_0^{\tau^\Psi\wedge\tau_h} 
\left(J_{\tau^{\Lambda,\ubar{x},\bar{x}}}(X_{\tau_h})	- \frac{\sqrt{\E_x(\tau_h)}\sigma(x)J_{\tau^{\Lambda,\ubar{x},\bar{x}}}(x)}{\E_x\left(l_{\tau^\Psi\wedge\tau_h}^x\right)}\right) dl_t^x\right)}{\sqrt{\E_x(\tau_h)}}\right|
\\&\enskip\leq
\psi_i \sup_{y\in[x-h,x+h]}\left|J_{\tau^{\Lambda,\ubar{x},\bar{x}}}(y)-J_{\tau^{\Lambda,\ubar{x},\bar{x}}}(x)\frac{\sigma(x) \sqrt{\E_x(\tau_h)}  }{\E_x\left(l_{\tau^\Psi\wedge\tau_h}^x\right)}\right|\left(\frac{\E_x(l_{\tau^\Psi\wedge\tau_h}^x)}{\sqrt{\E_x(\tau_h)}}\right)\rightarrow 0,
\end{align*}
as $h\searrow 0$. Hence, 
\begin{align*}
\lim_{h\searrow 0}\frac{(B2)}{\sqrt{\E_x(\tau_h)}}
=\psi_i\sigma(x)J_{\tau^{\Lambda,\ubar{x},\bar{x}}}(x). 
\end{align*}
Analogously, the limit of the last term in (B) satisfies 
\begin{align*}
\lim_{h\searrow 0}\frac{-\int_0^\infty \E_x\left(\int_0^{\tau^\Psi\wedge\tau_h}\psi_ie^{-rt}g(X_t)
dl_t^x\right)dF(r)}{\sqrt{\E_x(\tau_h)}}=-\psi_i\sigma(x)g(x).
\end{align*}
Adding of the limits 
yields, when recalling \eqref{J-in-terms-w}, the first result.
The second result can be proved by splitting up the expression in the limit according to \eqref{split-up-NE-cond} and then applying the first result separately to the split up terms, which relies on the observation that
\begin{align} \label{sec_resul_lambda}
\tau^{\Lambda,\ubar{x},\bar{x}} \diamond \tau^{\Lambda,\ubar{x},\bar{x}}(h) 
=\tau^{\Lambda,\ubar{x},\bar{x}}.
\end{align} 
\end{proof}

\begin{lemma} \label{main_lemma}  For any $\tau^{\Lambda,\ubar{x},\bar{x}},\tau^{\Psi,\mathcal{X},D} \in \mathcal{N}$ and $x \in D \backslash  \mathcal{X}$, it holds that
\begin{align*}
\lim_{h\searrow 0} \frac{J_{\tau^{\Lambda,\ubar{x},\bar{x}}\circ \theta_{\tau_h}+\tau_h}(x)-J_{\tau^{\Lambda,\ubar{x},\bar{x}} \diamond \tau^{\Psi,\mathcal{X},D}(h)}(x)}{\E_x(\tau_h)}
=\frac{1}{2}(\psi(x-)+\psi(x))(J_{\tau^{\Lambda,\ubar{x},\bar{x}}}(x)-g(x)).
\end{align*}
 Moreover, for $x\in (\alpha,\bar{x}) \backslash \{\ubar{x}\}$, 
\begin{align*}
\lim_{h\searrow 0}\frac{J_{\tau^{\Lambda,\ubar{x},\bar{x}}\circ \theta_{\tau_h}+\tau_h}(x)-J_{\tau^{\Lambda,\ubar{x},\bar{x}}}(x)}{\E_x(\tau_h)}=\frac{1}{2}(\lambda(x-)+\lambda(x))\left(J_{\tau^{\Lambda,\ubar{x},\bar{x}}}(x)-g(x)\right).
\end{align*}
\end{lemma}
The proof of Lemma \ref{main_lemma} is relegated to the appendix since it relies on arguments similar to those in the proof of Lemma \ref{main_lemma_local}.

We are now ready to present four propositions which will allow us to verify that the conditions \eqref{mainthmcond1}--\eqref{mainthmcond4} are sufficient for the equilibrium condition \eqref{eqdef2} to hold. 
\begin{proposition}\label{prop_CD_local}
For any $\tau^{\Lambda,\ubar{x},\bar{x}},\tau^{\Psi,\mathcal{X},D} \in \mathcal{N}$ and $x\in (\alpha,\ubar{x})\cap \mathcal{X}$ satisfying 
\begin{align*}
J_{\tau^{\Lambda,\ubar{x},\bar{x}}}(x)<g(x),
\end{align*}
 it holds that
\begin{align*}
 \lim_{h\searrow 0}\frac{J_{\tau^{\Lambda,\ubar{x},\bar{x}}}(x)-J_{\tau^{\Lambda,\ubar{x},\bar{x}}\diamond \tau^{\Psi,\mathcal{X},D}(h)}(x)}{\E_x(\tau_h)}=-\infty.
\end{align*}
\end{proposition}
\begin{proof}
By Lemma \ref{main_lemma_local}, we have that
\begin{align*}
 \lim_{h\searrow 0}\frac{J_{\tau^{\Lambda,\ubar{x},\bar{x}}}(x)-J_{\tau^{\Lambda,\ubar{x},\bar{x}}\diamond \tau^{\Psi,\mathcal{X},D}(h)}(x)}{\sqrt{\E_x(\tau_h)}}  
& = -\sigma(x)\left(\sum_{i=1}^nI_{\{x=x_i\}}\psi_i\right)\left(g(x)-J_{\tau^{\Lambda,\ubar{x},\bar{x}}}(x)\right)\\
& < 0.
\end{align*}
The result follows.
\end{proof}

\begin{proposition}\label{prop_CD} 
For any $\tau^{\Lambda,\ubar{x},\bar{x}},\tau^{\Psi,\mathcal{X},D} \in \mathcal{N}$ and 
$x \in (\alpha,\bar{x}) \cap D \backslash (\{\ubar{x}\}\cup \mathcal{X})$, it holds that
\begin{align*}
\lim_{h\searrow 0}\frac{J_{\tau^{\Lambda,\ubar{x},\bar{x}}}(x)-J_{\tau^{\Lambda,\ubar{x},\bar{x}}\diamond \tau^{\Psi,\mathcal{X},D}(h)}(x)}{\E_x(\tau_h)} = \frac{1}{2}(\lambda(x-)+\lambda(x)-\psi(x-)-\psi(
x))\left(g(x)-J_{\tau^{\Lambda,\ubar{x},\bar{x}}}(x)\right).
\end{align*}
\end{proposition}  
\begin{proof}
This can be shown by splitting up the limit according to \eqref{split-up-NE-cond} and using Lemma \ref{main_lemma}.
\end{proof}

\begin{proposition}\label{prop_CcD} For any $\tau^{\Lambda,\ubar{x},\bar{x}},\tau^{\Psi,\mathcal{X},D} \in \mathcal{N}$ and $x\in (\bar{x},\beta)\cap D$, 
\begin{align*}
\lim_{h\searrow 0}\frac{J_{\tau^{\Lambda,\ubar{x},\bar{x}}}(x)-J_{\tau^{\Lambda,\ubar{x},\bar{x}}\diamond \tau^{\Psi,\mathcal{X},D}(h)}(x)}{\E_x(\tau_h)}
= -f(x) - \mathbf{A}g(x) +g(x)\int_0^\infty rdF(r).
\end{align*}
\end{proposition}

\begin{proof}
Since $D$ and $(\bar{x},\beta)$ are open it follows that there exists a constant $\bar{h}>0$ such that, for $0<h\leq \bar{h}$, 
\begin{align*}
\tau^{\Lambda,\ubar{x},\bar{x}} \diamond \tau^{\Psi,\mathcal{X},D}(h) 
& = I_{\{\tau^{\Psi,\mathcal{X},D} \leq \tau_h\}}\tau^{\Psi,\mathcal{X},D} + I_{\{\tau^{\Psi,\mathcal{X},D} > \tau_h\}}({\tau^{\Lambda,\ubar{x},\bar{x}}}\circ \theta_{\tau_h}+\tau_h)\\
& = I_{\{\tau^{\Psi,\mathcal{X}} \leq \tau_h\}}\tau^{\Psi,\mathcal{X}} + I_{\{\tau^{\Psi,\mathcal{X}} > \tau_h\}}({\tau^{\Lambda,\ubar{x},\bar{x}}}\circ \theta_{\tau_h}+\tau_h)\\
&= \tau^{\Psi,\mathcal{X}} \wedge \tau_h.
\end{align*}
Since in this case $J_{\tau^{\Lambda,\ubar{x},\bar{x}}}(x)=g(x)$ it directly follows that
\begin{align*}
J_{\tau^{\Lambda,\ubar{x},\bar{x}}}&(x)-J_{\tau^{\Lambda,\ubar{x},\bar{x}}\diamond \tau^{\Psi,\mathcal{X}}(h)}(x)\\
&=g(x)-\E_x(h( \tau^{\Psi,\mathcal{X}} \wedge \tau_h)g(X_{\tau^{\Psi,\mathcal{X}}\wedge \tau_h}))-\E_x\left(\int_{0}^{ \tau^{\Psi,\mathcal{X}} \wedge \tau_h} h(s)f(X_s)ds\right).
\end{align*}
Using Lemma \ref{prop_limit_local_time_time_tau-h} and arguments similar to those in the proof of Lemma \ref{main_lemma} we find for the last term above that 
\begin{align*}
\lim_{h\searrow 0}\frac{-\E_x\left(\int_{0}^{ \tau^{\Psi,\mathcal{X}} \wedge \tau_h} h(s)f(X_s)ds\right)}{\E_x(\tau_h)}
%
%
=-f(x).
\end{align*}
For the remaining parts we use It\^{o}'s formula and arguments similar to those above to obtain
\begin{align*}
\frac{g(x)-\E_x(h( \tau^{\Psi,\mathcal{X}} \wedge \tau_h)g(X_{\tau^{\Psi,\mathcal{X}}\wedge \tau_h}))}{\E_x(\tau_h)}
&=\int_0^\infty\frac{-\E_x\left(\int_{0}^{ \tau^{\Psi,\mathcal{X}} \wedge \tau_h}e^{-rs}(\mathbf{A}-r)g(X_s)ds\right)}{\E_x(\tau_h)}dF(r)\\
&\rightarrow g(x)\int_0^\infty rdF(r)-\mathbf{A}g(x),
\end{align*}
as $h\searrow 0$. Adding the two limits completes the proof.
\end{proof}


\begin{proposition}\label{prop_lower_border} Consider an arbitrary $\tau^{\Lambda,\ubar{x},\bar{x}}\in \mathcal{N}$ and suppose that $J_{\tau^{\Lambda,\ubar{x},\bar{x}}}$ satisfies the differentiability condition in Theorem \ref{main_thm}, condition \eqref{mainthmcond2}  and the smooth fit condition \eqref{mainthmcond4}. Consider an arbitrary $\tau^{\Psi,\mathcal{X},D} \in \mathcal{N}$. Then, for $x \in [\ubar{x},\bar{x}]\cap D$, it holds that
\begin{align*}
\lim_{h\searrow 0}\frac{J_{\tau^{\Lambda,\ubar{x},\bar{x}}\circ \theta_{\tau_h}+\tau_h}(x)-J_{\tau^{\Lambda,\ubar{x},\bar{x}}\diamond \tau^{\Psi,\mathcal{X},D}(h)}(x)}{\E_x(\tau_h)} =0.
\end{align*}
Furthermore, for $x\in [\ubar{x},\bar{x})\cap D$ (note that this set is empty for pure strategies), it holds that
\begin{align*}
\lim_{h\searrow 0}\frac{J_{\tau^{\Lambda,\ubar{x},\bar{x}}}(x)-J_{\tau^{\Lambda,\ubar{x},\bar{x}}\diamond \tau^{\Psi,\mathcal{X},D}(h)}(x)}{\E_x(\tau_h)} =0.
\end{align*} 
\end{proposition}
\begin{proof}

By splitting up the term in the second limit using \eqref{split-up-NE-cond} and then using \eqref{sec_resul_lambda}, we find that the second result follows directly from the first result. Let us prove the first result. 
%
%
%
%
Similarly to the proof of Lemma \ref{main_lemma_local} we find that
\begin{align*}
&\frac{J_{\tau^{\Lambda,\ubar{x},\bar{x}}\circ \theta_{\tau_h}+\tau_h}(x)-J_{\tau^{\Lambda,\ubar{x},\bar{x}}\diamond \tau^{\Psi}(h)}(x)}{\E_x(\tau_h)}=\\ 
&\frac{\int_0^\infty\E_{x}\left(I_{\{\tau^{\Psi} \leq \tau_h\}}\left((e^{-r\tau_h}-1)w_{\tau^{\Lambda,\ubar{x},\bar{x}}}(X_{\tau_h},r)+\int_{\tau^{\Psi}}^{\tau_h} e^{-rs}f(X_s)ds +(1-e^{-r\tau^{\Psi}})g(X_{\tau^{\Psi}})\right)\right)dF(r)}{\E_{x}\left(\tau_h\right)}\\
&\enskip +\frac{\E_{x}\left(I_{\{\tau^{\Psi} \leq \tau_h\}}\left(J_{\tau^{\Lambda,\ubar{x},\bar{x}}}(X_{\tau_h})-g(X_{\tau^{\Psi}})\right)\right)}{\E_{x}\left(\tau_h\right)}\\
%
%
&=(A)+(B).  
\end{align*}
Analogously to the proof of Lemma \ref{main_lemma_local}, we condition on the whole path of $X$ and use $|e^{-y}-1|\leq y$, for $y\geq 0$, to obtain 
\begin{align*}
\lim_{h\searrow 0}|(A)|&\leq \lim_{h\searrow 0}\int_0^\infty\E_x\left(\int_0^{\tau^\Psi\wedge\tau_h}\left(r\tau_hw_{\tau^{\Lambda,\ubar{x},\bar{x}}}(X_{\tau_h},r)+\int_{t}^{\tau_h}f(X_s)ds+rtg(X_t)
\right)d\Psi_t\right)dF(r)\\
&\leq \lim_{h\searrow 0} \int_0^\infty\left(\sup_{y_1,y_2,y_3\in[x-h,x+h]}(rw_{\tau^{\Lambda,\ubar{x},\bar{x}}}(y_1,r)+f(y_2) +rg(y_3))\frac{\E_x\left(\tau_h\Psi_{\tau_h}\right)}{\E_x(\tau_h)}\right)dF(r)=0,
\end{align*}
where the limit can be found by writing $\Psi_{\tau_h }= \psi_iI_{\{x=x_i\}} l^x_{\tau_h} + \int_0^{\tau_h}\psi(X_s)ds$ and then using Lemma \ref{lemma-tau-h-oct} and Lemma \ref{prop_limit_local_time_time_tau-h}.

Using the local time-space formula, see, e.g., \cite[p. 75]{Peskir}, Proposition \ref{centralProposition} and that $J_{\tau^{\Lambda,\ubar{x},\bar{x}}}(x)=g(x)$ for $x\in [\ubar{x},\bar{x}]$  we obtain, for any sufficiently small $h$, 
\begin{align*}
(B)
&=\frac{\E_{x}\left(I_{\{\tau^{\Psi} \leq \tau_h\}}\left(J_{\tau^{\Lambda,\ubar{x},\bar{x}}}(X_{\tau_h})-g(x)+g(x)-g(X_{\tau^{\Psi}})\right)\right)}{\E_{x}\left(\tau_h\right)}\\
%
%
%
&=\frac{\E_{x}\left(\int_0^{\tau^{\Psi}\wedge \tau_h}\left(\int_0^{\tau_h}\mathbf{A}J_{\tau^{\Lambda,\ubar{x},\bar{x}}}(X_s)I_{\{X_s\neq \ubar{x}\}}ds-\int_0^t\mathbf{A}g(X_s)ds\right)d\Psi_t\right)}{\E_x\left(\tau_h\right)}\\
&\enskip+\frac{\E_{x}\left(\int_0^{\tau^{\Psi}\wedge \tau_h} \int_0^{\tau_h}\sigma(X_s)(J'_{\tau^{\Lambda,\ubar{x},\bar{x}}}(X_s)-g'(X_s))dW_s d\Psi_t+I_{\{\tau^\Psi\leq \tau_h\}}\int_{\tau^{\Psi}}^{\tau_h}\sigma(X_s)g'(X_s)dW_s\right)}{\E_x\left(\tau_h\right)}.
\end{align*}
%
%
By Cauchy-Schwartz, It\^{o} isometry and smooth fit \eqref{mainthmcond4}, we obtain for the first part of the last term above that
\begin{align*}
&\left|\frac{\E_{x}\left(
\int_0^{\tau^{\Psi}\wedge \tau_h}
 \int_0^{\tau_h}\sigma(X_s)(J'_{\tau^{\Lambda,\ubar{x},\bar{x}}}(X_s)-g'(X_s))dW_s d\Psi_t\right)}{\E_x\left(\tau_h\right)}\right|\\&\enskip=
\left|\frac{\E_x\left(\Psi_{\tau^\Psi\wedge \tau_h}
 \int_0^{\tau_h}\sigma(X_s)(J'_{\tau^{\Lambda,\ubar{x},\bar{x}}}(X_s)-g'(X_s))dW_s \right)}{\E_x\left(\tau_h\right)}\right|\\&\enskip\leq \sqrt{\frac{\E_{x}\left(
\Psi_{\tau_h}^2\right)}
{\E_x\left(\tau_h\right)}\frac{\E_{x}\left(
\int_0^{\tau_h}\sigma^2(X_s)(J'_{\tau^{\Lambda,\ubar{x},\bar{x}}}(X_s)-g'(X_s))^2ds\right)}{\E_x\left(\tau_h\right)}}\rightarrow 2 c\sigma(x) |J'_{\tau^{\Lambda,\ubar{x},\bar{x}}}(x)-g'(x)|=0,
\end{align*}
as $h\searrow 0$, where $c$ is a constant (to see this use the same argument used when finding the previous limit). For the second part of the term, we obtain by conditioning on $\mathcal{F}_{\tau^\Psi}$ and the Markov property that
\begin{align*}
\E_{x}\left(
I_{\{\tau^\Psi\leq \tau_h\}}\int_{\tau^\Psi}^{\tau_h}\sigma(X_s)g'(X_s)dW_s\right)
&=\E_{x}\left(
I_{\{\tau^{\Psi} \leq \tau_h\}}\E_x\left(\int_{\tau^\Psi}^{\tau_h}\sigma(X_s)g'(X_s)dW_s\bigg|\mathcal{F}_{\tau^\Psi}\right)\right)\\
&=\E_{x}\left(
I_{\{\tau^{\Psi} \leq \tau_h\}}\E_{X_{\tau^\Psi}}\left(\int_{0}^{\tau_h}\sigma(X_s)g'(X_s)dW_s\right)\right)\\&=0.
\end{align*} 
Finally, we obtain for (B) 
\begin{align*}
\lim_{h\searrow 0}|(B)|\leq\lim_{h\searrow 0} \sup_{y_1\in[x-h,x+h]\setminus\{\ubar{x}\}}(|\mathbf{A} J_{\tau^{\Lambda,\ubar{x},\bar{x}}}(y_1)|+|\mathbf{A}g(y_2)|)\frac{\E_x\left(\tau_h \Psi_{\tau_h}\right)}{\E_x\left(\tau_h\right)}=0,
\end{align*}
where the limit is found similarly to the previous limits; using also that for $x<\ubar{x}$ it holds that
$\mathbf{A}  J_{\tau^{\Lambda,\ubar{x},\bar{x}}}(x) 
= \int_0^\infty\mathbf{A} w_{\tau^{\Lambda,\ubar{x},\bar{x}}}(x,r) dF(r)
= -f(x)+\int_0^\infty r w_{\tau^{\Lambda,\ubar{x},\bar{x}}}(x,r)dF(r)$, where  the first equality immediately holds (here $\mathbf{A}$ is interpreted as the characteristic operator) and the second last equality holds by standard martingale arguments (specifically in this case \eqref{prop_help_diff}, below, holds), and for $x>\ubar{x}$ it holds that $\mathbf{A} J_{\tau^{\Lambda,\ubar{x},\bar{x}}}(x) = \mathbf{A} g(x)$ by \eqref{mainthmcond2}. 
\end{proof} 
\begin{proposition}\label{prop_pCD}
Consider an arbitrary $\tau^{\Lambda,\ubar{x},\bar{x}}\in \mathcal{N}$ and suppose all conditions of Theorem \ref{main_thm} are satisfied. Consider an arbitrary $\tau^{\Psi,\mathcal{X},D} \in \mathcal{N}$. Then, if $\bar{x}\in D$, it holds that
\begin{align}\label{prop_pCD_eq}
\limsup_{h\searrow 0}\frac{J_{\tau^{\Lambda,\ubar{x},\bar{x}}}(\bar{x})-J_{\tau^{\Lambda,\ubar{x},\bar{x}}\diamond \tau^{\Psi,\mathcal{X},D}(h)}(\bar{x})}{\E_x(\tau_h)} 
\leq 0.
\end{align}
\end{proposition}

\begin{proof}
As usual we split up the expression in the limit in two terms according to \eqref{split-up-NE-cond}. The second term converges to zero by the first result in Proposition \ref{prop_lower_border}. In the proof we rely on the observation that $J_{\tau^{\Lambda,\ubar{x},\bar{x}}}(\bar{x})= g(\bar{x})$. We handle the first term in two cases. First, we consider a mixed threshold strategy (i.e., $\ubar{x}<\bar{x}$). Then we find for the first term that
\begin{align*}
J_{\tau^{\Lambda,\ubar{x},\bar{x}}\circ \theta_{\tau_h}+\tau_h}(\bar{x})
&=\int_0^\infty\E_{\bar{x}}\left(\int_0^{\tau_h}e^{-rs}f(X_s)ds+e^{-r\tau_h}w_{\tau^{\Lambda,\ubar{x},\bar{x}}}(X_{\tau_h},r)\right)dF(r)\\
&=\int_0^\infty\E_{\bar{x}}\left(\int_0^{\tau_h}e^{-rs}f(X_s)ds+(e^{-r\tau_h}-1)w_{\tau^{\Lambda,\ubar{x},\bar{x}}}(X_{\tau_h},r)\right)dF(r)\\
&\quad+\E_{\bar{x}}\left(J_{\tau^{\Lambda,\ubar{x},\bar{x}}}(X_{\tau_h})\right).
\end{align*}
By \eqref{mainthmcond2} there exists a $\bar{h}$, such that $J_{\tau^{\Lambda,\ubar{x},\bar{x}}}(x)= g(x)$ for $x\in[\bar{x}-\bar{h},\bar{x}+\bar{h}]$ and using It\^{o}'s formula we thus obtain
for $h\leq \bar{h}$,
\begin{align*}
g(\bar{x})-J_{\tau^{\Lambda,\ubar{x},\bar{x}}\circ \theta_{\tau_h}+\tau_h}(\bar{x})
&= \int_0^\infty\E_{\bar{x}}\left(-\int_0^{\tau_h}\left(e^{-rs}f(X_s)+\mathbf{A}g(X_s)\right)ds+(1-e^{-r\tau_h})w_{\tau^{\Lambda,\ubar{x},\bar{x}}}(X_{\tau_h},r)\right)dF(r). 
\end{align*}
We finish the proof for the mixed case by showing that the expression above converges to the negative of the term given in \eqref{mainthmcond3}. For the first integral part we obtain 
\begin{align*}
  -\frac{\int_0^\infty\E_{\bar{x}}\left(\int_0^{\tau_h}\left(e^{-rs}f(X_s)+\mathbf{A}g(X_s)\right)ds\right)dF(r)}{\E_{\bar{x}}(\tau_h)}\rightarrow -f(\bar{x})-\mathbf{A}g(\bar{x}),
 \end{align*}
as  $h\searrow 0$.
%
%
For the second part we obtain using $|1-e^{-y}|\leq y$ for $y\geq 0$ and basic calculations that
\begin{align*}
  & \left| \frac{\int_0^\infty\E_{\bar{x}}\left((1-e^{-r\tau_h})w_{\tau^{\Lambda,\ubar{x},\bar{x}}}(X_{\tau_h},r)\right)dF(r)}{\E_{\bar{x}}(\tau_h)}-g(\bar{x})\int_0^\infty rdF(r)\right|\\&\enskip= \left| \frac{\int_0^\infty\E_{\bar{x}}\left((1-e^{-r\tau_h})(w_{\tau^{\Lambda,\ubar{x},\bar{x}}}(X_{\tau_h},r)-g(\bar{x}))+rg(\bar{x})\int_0^{\tau_h}(e^{-rs}-1)ds\right)dF(r)}{\E_{\bar{x}}(\tau_h)}\right|\\
  &\enskip\leq
 \int_0^\infty\left(\sup_{y\in[\bar{x}-h,\bar{x}+h]}r|w_{\tau^{\Lambda,\ubar{x},\bar{x}}}(y,r)-g(\bar{x})|+rg(\bar{x})\left|\frac{\E_{\bar{x}}\left(\int_0^{\tau_h}(e^{-rs}-1)ds\right)}{\E_{\bar{x}}(\tau_h)}\right|\right)dF(r)\rightarrow 0
\end{align*}
as $h\searrow 0$, where the limit is found using the usual arguments as well as $w_{\tau^{\Lambda,\ubar{x},\bar{x}}}(\bar{x},r)=g(\bar{x})$ (which holds trivially).

We now consider a pure threshold strategy (i.e., $\ubar{x}=\bar{x})$. In this case we note that by the standard line of arguments it holds for $x<\bar{x}$ that $x \mapsto w_{\tau^{\Lambda,\ubar{x},\bar{x}}}(x,r)$ satisfies the differential equation 
\begin{align}\label{prop_help_diff}
(\mathbf{A}-r)w_{\tau^{\Lambda,\ubar{x},\bar{x}}}(x,r)+f(x)=0.
\end{align}
and $w_{\tau^{\Lambda,\ubar{x},\bar{x}}}(x,r)=g(x)$, for $x\geq \bar{x}$. Using these observations and the local time-space formula 
we obtain similarly to the above that 
\begin{align*}
g(\bar{x})-J_{\tau^{\Lambda,\ubar{x},\bar{x}}\circ \theta_{\tau_h}+\tau_h}(\bar{x})
&= \int_0^\infty\E_{\bar{x}}\left(-\int_0^{\tau_h}\left(e^{-rs}f(X_s)+e^{-rs}(\mathbf{A}-r)w_{\tau^{\Lambda,\ubar{x},\bar{x}}}(X_s,r)I_{\{X_s\neq \bar{x}\}}\right)ds\right)dF(r)\\
&\enskip-\int_0^\infty\E_{\bar{x}}\left(\frac{1}{2}\int_0^{\tau_h}e^{-rs}(g'(\bar{x}+)-w'_{\tau^{\Lambda,\ubar{x},\bar{x}}}(\bar{x}-,r))dl^{\bar{x}}_t\right)dF(r).
\end{align*} 
Using smooth fit \eqref{mainthmcond4} we obtain for the second term, with some calculations (using, e.g., $\int_0^\infty w'_{\tau^{\Lambda,\ubar{x},\bar{x}}}(\bar{x}-,r)dF(r)= g'(\bar{x}+)$), that 
\begin{align*}
&\left|\frac{\int_0^\infty\E_{\bar{x}}\left(\frac{1}{2}\int_0^{\tau_h}e^{-rs}(g'(\bar{x}+)-w'_{\tau^{\Lambda,\ubar{x},\bar{x}}}(\bar{x}-,r))dl^{\bar{x}}_t\right)dF(r)}{E_{\bar{x}}(\tau_h)}\right|\\
&\enskip=\left|\frac{\int_0^\infty\E_{\bar{x}}\left(\frac{1}{2}\int_0^{\tau_h}(1-e^{-rs})(g'(\bar{x}+)-w'_{\tau^{\Lambda,\ubar{x},\bar{x}}}(\bar{x}-,r))dl^{\bar{x}}_t\right)dF(r)}{E_{\bar{x}}(\tau_h)}\right|\\&\enskip \leq\left|\int_0^\infty r(g'(\bar{x}+)-w'_{\tau^{\Lambda,\ubar{x},\bar{x}}}(\bar{x}-,r))dF(r)\right|\frac{E_{\bar{x}}(\tau_hl_{\bar{x}}^{\tau_h})}{E_{\bar{x}}(\tau_h)}\rightarrow 0
\end{align*}
as $h\searrow 0$. For the first term we note that using, e.g, \eqref{prop_help_diff}, we obtain (for $h<\bar{h}$)
 \begin{align*}
&\int_0^\infty\E_{\bar{x}}\left(-\int_0^{\tau_h}\left(e^{-rs}f(X_s)+e^{-rs}(\mathbf{A}-r)w_{\tau^{\Lambda,\ubar{x},\bar{x}}}(X_s,r)I_{\{X_s\neq \bar{x}\}}\right)ds\right)dF(r)\\
&\enskip=-\int_0^\infty\E_{\bar{x}}\left(\int_0^{\tau_h}e^{-rs}\left(f(X_s)I_{\{X_s\geq \bar{x}\}}+(\mathbf{A}-r)w_{\tau^{\Lambda,\ubar{x},\bar{x}}}(X_s,r)I_{\{X_s> \bar{x}\}}\right)ds\right)dF(r)\\
&\enskip= \int_0^\infty\E_{\bar{x}}\left(\int_0^{\tau_h}(1-e^{-rs})\left(f(X_s)I_{\{X_s\geq \bar{x}\}}+(\mathbf{A}-r)g(X_s)I_{\{X_s> \bar{x}\}}\right)ds\right)dF(r)\\
&\enskip\enskip- \E_{\bar{x}}\left(\int_0^{\tau_h}f(X_s)I_{\{X_s\geq \bar{x}\}}+\left(\mathbf{A}g(X_s)-g(X_s)\int_0^\infty r dF(r)\right)I_{\{X_s> \bar{x}\}}ds\right)\\
&\enskip= (A)+(B).
\end{align*}
For $(A)$, we find, using the usual arguments, that 
\begin{align*}
\frac{(A)}{\E_{\bar{x}}(\tau_h)}&\leq\frac{ \int_0^\infty\E_{\bar{x}}\left(\int_0^{\tau_h}rs
|f(X_s)I_{\{X_s\geq \bar{x}\}}+(\mathbf{A}-r)g(X_s)I_{\{X_s> \bar{x}\}}|
ds\right)dF(r)}{\E_{\bar{x}}(\tau_h)}\\
&\leq 
\sup_{y_1,y_2,y_3\in [\bar{x},\bar{x}+h]}
\left(
\int_0^\infty rdF(r)f(y_1)+
|\mathbf{A}g(y_2)|
+\int_0^\infty r^2dF(r)g(y_3)
\right)
\frac{\E_{\bar{x}}(\tau_h^2)}{\E_{\bar{x}}(\tau_h)}\rightarrow 0
\end{align*}
as $h\searrow 0$. For $(B)$, we obtain, using $f(x)\geq 0$ and \eqref{mainthmcond3}, 
\begin{align*}
\E_{\bar{x}}\left(\int_0^{\tau_h}f(X_s)I_{\{X_s\geq \bar{x}\}}+\left(\mathbf{A}g(X_s)-g(X_s)\int_0^\infty r dF(r)\right)I_{\{X_s> \bar{x}\}}ds\right)\geq 0
\end{align*}
for all $h<\bar{h}$. This gives us
\begin{align*}
\limsup_{h\searrow 0}\frac{-\E_{\bar{x}}\left(\int_0^{\tau_h}e^{-rs}\left(f(X_s)I_{\{X_s\geq \bar{x}\}}+\left(\mathbf{A}g(X_s)-g(X_s)\int_0^\infty rdF(r)\right)I_{\{X_s> \bar{x}\}}\right)ds\right)}{\E_{\bar{x}}(\tau_h)}\leq 0.
\end{align*}
The result follows.
\end{proof}

\begin{proof} (of Theorem \ref{main_thm}) 
We will show that the equilibrium condition \eqref{eqdef2} is satisfied for each $\tau^{\Psi,\mathcal{X},D}\in \mathcal{N}$ and each $x \in (\alpha,\beta)$ by considering different cases for $x$. For $x\in D$: 
\begin{itemize} 
\item  If $x \in (\alpha,\ubar{x})$, then the left side of \eqref{eqdef2} is, by Proposition \ref{prop_CD_local} and Proposition \ref{prop_CD}, smaller than zero by condition \eqref{mainthmcond1} (where we also used that $\lambda(x)=0$ for $x \in (\alpha,\ubar{x})$). Hence, \eqref{eqdef2} holds. 

\item  If $x \in [\ubar{x},\bar{x})$, then the left side of \eqref{eqdef2} is using Proposition \ref{prop_lower_border} equal to 0 by \eqref{mainthmcond2} and \eqref{mainthmcond4}. 
Hence, \eqref{eqdef2} holds. 

\item If $x \in [\bar{x},\beta)$, then Proposition \ref{prop_CcD} (together  with \eqref{mainthmcond3}) and Proposition \ref{prop_pCD} implies that  \eqref{eqdef2} holds. 

\end{itemize}
For $x\notin D$, the numerator in the left side of \eqref{eqdef2} is $J_{\tau^{\Lambda,\ubar{x},\bar{x}}}(x) - g(x)$ and hence \eqref{mainthmcond1} and \eqref{mainthmcond2}  implies that  \eqref{eqdef2} holds for $x\in (\alpha,\bar{x})$. In the case that $x\geq \bar{x}$ the numerator is zero. 
\end{proof}

\subsection{Further results for mixed threshold strategies}\label{sec:mixed-th-strat}
In this section we derive further results for mixed threshold strategies $\tau^{\Lambda,\ubar{x},\bar{x}}$, under certain differentiability assumptions. In particular, Theorem \ref{theorem_ansatz}, see also Remark \ref{ansatz_remark}, establishes a differential equation for
the intensity functions $\lambda(\cdot)$ and the cost functions $x\mapsto w_{\tau^{\Lambda,\ubar{x},\bar{x}}}(x,r)$, which is satisfied by $J_{\tau^{\Lambda,\ubar{x},\bar{x}}}(\cdot)$; while Theorem \ref{local_ansatz} finds that $x\mapsto w_{\tau^{\Lambda,\ubar{x},\bar{x}}}(x,r)$ generally has a jump in its derivative at the local time push point $\ubar{x}$ and gives an expression for this jump. We use these results to find an equilibrium candidate for problem \eqref{cost-func-real-op} in Section \ref{sec:real-option-problem}.

\begin{theorem}\label{theorem_ansatz}
Let $\tau^{\Lambda,\ubar{x},\bar{x}} \in \cal N$ be a mixed threshold stopping time with a stopping intensity function $\lambda(\cdot)$ that is continuous on $(\ubar{x},\bar{x})$. Suppose each function $x \mapsto w_{\tau^{\Lambda,\ubar{x},\bar{x}}}(x,r)$ is twice continuously differentiable in $x$ on $(\ubar{x},\bar{x})$. Then, for $x\in (\ubar{x},\bar{x})$, it holds that
\begin{align}\label{ansatz_eq}
f(x)+\mathbf{A}J_{\tau^{\Lambda,\ubar{x},\bar{x}}}(x)-\int_0^\infty rw_{\tau^{\Lambda,\ubar{x},\bar{x}}}(x,r)dF(r)=\lambda(x)\left(J_{\tau^{\Lambda,\ubar{x},\bar{x}}}(x)-g(x)\right).
\end{align}
\end{theorem}
\begin{proof}
By Lemma \ref{main_lemma}  we find that 
\begin{align*}
\lim_{h\searrow 0}\frac{J_{\tau^{\Lambda,\ubar{x},\bar{x}}\circ \theta_{\tau_h}+\tau_h}(x)-J_{\tau^{\Lambda,\ubar{x},\bar{x}}}(x)}{\E_x(\tau_h)}=\lambda(x)\left(J_{\tau^{\Lambda,\ubar{x},\bar{x}}}(x)-g(x)\right). 
\end{align*}
On the other hand we obtain by Lemma \ref{mp_lemma} and Itô's formula that
\begin{align*}
& J_{\tau^{\Lambda,\ubar{x},\bar{x}}\circ \theta_{\tau_h}+\tau_h}(x)\\
&\enskip =\int_0^\infty\E_x\left(  \int_0^{\tau_h}e^{-rs}f(X_s)ds+e^{-r\tau_h}w_{\tau^{\Lambda,\ubar{x},\bar{x}}}(X_{\tau_h},r)\right)dF(r)\\
&\enskip =\int_0^\infty\E_x\left(  \int_0^{\tau_h}e^{-rs}f(X_s)ds+w_{\tau^{\Lambda,\ubar{x},\bar{x}}}(x,r)+\int_0^{\tau_h}e^{-rt}(\mathbf{A}-r)w_{\tau^{\Lambda,\ubar{x},\bar{x}}}(X_t,r)dt\right)dF(r)\\
&\enskip =J_{\tau^{\Lambda,\ubar{x},\bar{x}}}(x)+\int_0^\infty\E_x\left(  \int_0^{\tau_h}e^{-rs}f(X_s)ds+\int_0^{\tau_h}e^{-rt}(\mathbf{A}-r)w_{\tau^{\Lambda,\ubar{x},\bar{x}}}(X_t,r)dt\right)dF(r).
\end{align*}
Hence, using the usual arguments from the proofs of the results in the previous section we find 
\begin{align*}
\lim_{h\searrow 0}\frac{J_{\tau^{\Lambda,\ubar{x},\bar{x}}\circ \theta_{\tau_h}+\tau_h}(x)-J_{\tau^{\Lambda,\ubar{x},\bar{x}}}(x)}{\E_x(\tau_h)}=f(x)+\mathbf{A}J_{\tau^{\Lambda,\ubar{x},\bar{x}}}(x)-\int_0^\infty rw_{\tau^{\Lambda,\ubar{x},\bar{x}}}(x,r)dF(r).
\end{align*}
\end{proof}

\begin{remark}\label{ansatz_remark} 
Suppose $\tau^{\Lambda,\ubar{x},\bar{x}}$ is a mixed threshold equilibrium stopping strategy such that 
the conditions of Theorem \ref {theorem_ansatz} as well as \eqref{mainthmcond1}--\eqref{mainthmcond4} are satified. Then, the corresponding functions $x \mapsto w_{\tau^{\Lambda,\ubar{x},\bar{x}}}(x,r)$ satisfy \eqref{ansatz_eq}, which in this case simplifies to
\begin{align*}
f(x)+\mathbf{A}g(x)-\int_0^\infty rw_{\tau^{\Lambda,\ubar{x},\bar{x}}}(x,r)dF(r)=0,
\end{align*}
for  $x \in (\ubar{x},\bar{x})$. This observation can be used to determine the functions $x \mapsto w_{\tau^{\Lambda,\ubar{x},\bar{x}}}(x,r)$ in the randomization region $(\ubar{x},\bar{x})$ when searching for an equilibrium; see the beginning of Section \ref{sec:candidate} for an example. 
\end{remark}

\begin{theorem}\label{local_ansatz}
Let $\tau^{\Lambda,\ubar{x},\bar{x}} \in \cal N$ be a mixed threshold stopping time and suppose, for a fixed $r>0$, that the function $x \mapsto w_{\tau^{\Lambda,\ubar{x},\bar{x}}}(x,r)$ is a member of 
$\mathcal{C}^2((\ubar{x}-h,\ubar{x})\cup (\ubar{x},\ubar{x}+h)) \cap \mathcal{C}(\alpha,\beta)$ for some $h>0$ and that its left and right derivatives exist. Then, 
\begin{align*}
2\lambda(w_{\tau^{\Lambda,\ubar{x},\bar{x}}}(\ubar{x},r)-g(\ubar{x}))=
w'_{\tau^{\Lambda,\ubar{x},\bar{x}}}(\ubar{x}+,r) - w'_{\tau^{\Lambda,\ubar{x},\bar{x}}}(\ubar{x}-,r).
\end{align*}
\end{theorem}
\begin{proof}
Using similar arguments as in the previous proofs (as well as Lemma \ref{lemma-tau-h-oct}) and the local time-space formula, we obtain 
\begin{align*}
&\lim_{h\searrow 0}\frac{w_{\tau^{\Lambda,\ubar{x},\bar{x}}\circ \theta_{\tau_h}+\tau_h}({\ubar{x}},r)-w_{\tau^{\Lambda,\ubar{x},\bar{x}}}(\ubar{x},r)}{\sqrt{\E_{\ubar{x}}(\tau_h)}}\\
&\enskip=\lim_{h\searrow 0}\frac{\E_{\ubar{x}}\left(
\int_0^{\tau_h}e^{-rs}f(X_s)ds+e^{-\tau_h}w_{\tau^{\Lambda,\ubar{x},\bar{x}}}(X_{\tau_h},r)
\right)-w_{\tau^{\Lambda,\ubar{x},\bar{x}}}(\ubar{x},r)}{\sqrt{\E_{\ubar{x}}(\tau_h)}}\\
&\enskip=\lim_{h\searrow 0}\frac{\E_{\ubar{x}}\left(
e^{-\tau_h}w_{\tau^{\Lambda,\ubar{x},\bar{x}}}(X_{\tau_h},r)
\right)-w_{\tau^{\Lambda,\ubar{x},\bar{x}}}(\ubar{x},r)}{\sqrt{\E_{\ubar{x}}(\tau_h)}}\\
&\enskip=\lim_{h\searrow 0}\frac{\E_{\ubar{x}}\left(\int_0^{\tau_h}e^{-rs}(\mathbf{A}-r)w_{\tau^{\Lambda,\ubar{x},\bar{x}}}(X_{s},r)I_{\{X_s\neq \ubar{x}\}}ds\right)}{\sqrt{\E_{\ubar{x}}(\tau_h)}}+\lim_{h\searrow 0}\frac{1}{2}\frac{\E_{\ubar{x}}\left(\int_0^{\tau_h}e^{-rs}\left(w'_{\tau^{\Lambda,\ubar{x},\bar{x}}}(\ubar{x}+,r)-w'_{\tau^{\Lambda,\ubar{x},\bar{x}}}(\ubar{x}-,r)\right)dl_{s}^{\ubar{x}}\right)}{\sqrt{\E_{\ubar{x}}(\tau_h)}}\\
&\enskip=
\frac{1}{2} \left(w'_{\tau^{\Lambda,\ubar{x},\bar{x}}}(\ubar{x}+,r)-w'_{\tau^{\Lambda,\ubar{x},\bar{x}}}(\ubar{x}-,r)\right)
\lim_{h\searrow 0}\frac{
\E_{\ubar{x}}\left(l_{\tau_h}^{\ubar{x}}\right)
}{\sqrt{\E_{\ubar{x}}(\tau_h)}}\\
&\enskip=\frac{1}{2}\left(w'_{\tau^{\Lambda,\ubar{x},\bar{x}}}(\ubar{x}+,r)-w'_{\tau^{\Lambda,\ubar{x},\bar{x}}}(\ubar{x}-,r)\right)\sigma(\ubar{x}).
\end{align*}
Furthermore, using Lemma \ref{main_lemma_local} and \eqref{sec_resul_lambda} we obtain 
\begin{align*}
\lim_{h\searrow 0}\frac{w_{\tau^{\Lambda,\ubar{x},\bar{x}}\circ \theta_{\tau_h}+\tau_h}({\ubar{x}},r)-w_{\tau^{\Lambda,\ubar{x},\bar{x}}}(\ubar{x},r)}{\sqrt{\E_{\ubar{x}}(\tau_h)}}=\lambda\sigma({\ubar{x}})\left(w_{\tau^{\Lambda,\ubar{x},\bar{x}}}(\ubar{x},r)-g({\ubar{x}})\right).
\end{align*}
The result follows.
\end{proof}

\section{The real options problem}\label{sec:real-option-problem}
In \cite{tan2021failure} a real options problem is studied using pure stopping strategies; see Section \ref{intro} for a description of this problem and the findings in that paper. In this section we study this problem using the theoretical framework for mixed stopping strategies developed in previous sections. In particular, we consider in this section the cost function \eqref{cost-func-real-op} and the discount function $h$ given by a mixture of two exponential discounting factors according to 
\begin{align}\label{disc-F-realOp}
h(t)=pe^{-r_1t}+(1-p)e^{-r_2t}, \enskip r_2>r_1>0, \enskip p\in (0,1).
\end{align}
This implies that the cost function can be written as
\begin{align*}
J_\tau(x)
& = \int_0^\infty w_{\tau^{\Lambda,\ubar{x},\bar{x}}}(x,r)dF(r)\\
& = pw_\tau(x,r_1) + (1-p)w_\tau(x,r_2)
\end{align*}
where
\begin{align} \label{w-def-realoption}
w_\tau(x,r_i)=\E_x\left(\int_0^\tau e^{-r_is}X_sds+e^{-r_i\tau}K\right).
\end{align}
As in \cite{tan2021failure}, we also suppose for simplicity that $X$ is a GBM without drift, i.e., 
\begin{align*}
dX_t = \sigma X_t dW_t, \enskip X_0 =x.
\end{align*}
In Section \ref{sec:candidate} we derive a candidate mixed threshold equilibrium stopping time $\tau^{\Lambda,\ubar{x},\bar{x}}$ for this problem using, e.g., the results in Section \ref{sec:mixed-th-strat}. 
In Section \ref{sec:verification} we verify that the equilibrium candidate is indeed an equilibrium using Theorem \ref{main_thm}. 
We repeatedly use that the state space of $X$ is   $(\alpha,\beta)=(0,\infty)$.  


\subsection{Derivation of a candidate mixed threshold equilibrium stopping time}\label{sec:candidate}
The content of the present section describes how to identify an equilibrium candidate and is mainly of motivational value; 
%
while rigorous results are presented in Section \ref{sec:verification}.
In particular, we derive conditions that must be satisfied by 
(i) the stopping intensity function $\lambda(\cdot)$, 
(ii) the interval $(\ubar x,\bar x)$ on which randomization occurs, and 
(iii) the local time push $\lambda>0$ occurring at $\ubar x$;   
in order for the associated mixed stopping threshold strategy $\tau^{\Lambda,\ubar{x},\bar{x}}$ to satisfy the conditions \eqref{mainthmcond1}--\eqref{mainthmcond4}. This gives us an equilibrium candidate. 

Suppose $\tau^{\Lambda,\ubar{x},\bar{x}}$ satisfies \eqref{mainthmcond1}--\eqref{mainthmcond4}  and is in particular a mixed threshold equilibrium. Using \eqref{w-def-realoption} and Remark \ref{ansatz_remark}, and supposing we have enough differentiability, we find that for $x\in (\ubar{x},\bar{x})$, it holds that  
\begin{align*}
x&=\int_0^\infty r w_{\tau^{\Lambda,\ubar{x},\bar{x}}}(x,r)dF(r)\\
& =pr_1w_{\tau^{\Lambda,\ubar{x},\bar{x}}}(x,r_1)+(1-p)r_2w_{\tau^{\Lambda,\ubar{x},\bar{x}}}(x,r_2)
\end{align*}
and, using \eqref{mainthmcond2}, that
\begin{align*}
K&=J_{\tau^{\Lambda,\ubar{x},\bar{x}}}(x)\\
 & =pw_{\tau^{\Lambda,\ubar{x},\bar{x}}}(x,r_1)+(1-p)w_{\tau^{\Lambda,\ubar{x},\bar{x}}}(x,r_2).
\end{align*}
These two equations correspond, for $x$ fixed, to a linear equation system in $w_{\tau^{\Lambda,\ubar{x},\bar{x}}}(x,r_1)$ and $w_{\tau^{\Lambda,\ubar{x},\bar{x}}}(x,r_2)$ whose solution is
\begin{align*}
\left(\begin{matrix}
w_{\tau^{\Lambda,\ubar{x},\bar{x}}}(x,r_1)\\
w_{\tau^{\Lambda,\ubar{x},\bar{x}}}(x,r_2)
\end{matrix}\right)=\left(\begin{matrix}
pr_1&(1-p)r_2\\
p&1-p
\end{matrix}\right)^{-1}
\left(\begin{matrix}
x\\
K
\end{matrix}\right)
\end{align*}
which implies that 
\begin{align}\label{w-as-linear}
w_{\tau^{\Lambda,\ubar{x},\bar{x}}}(x,r_i)=a_ix+b_i, \mbox{ for } x\in (\ubar{x},\bar{x}), \enskip i=1,2
\end{align}
where 
\begin{align*}
a_1=\frac{1}{p(r_1-r_2)},&\qquad b_1=\frac{-r_2K}{p(r_1-r_2)},\\
a_2=\frac{1}{(1-p)(r_2-r_1)},&\qquad b_2=\frac{-r_1K}{(1-p)(r_2-r_1)}.
\end{align*}
Moreover, it holds on $x\in (\ubar{x},\bar{x})$ that
\begin{align*}
x+\mathbf{A}w_{\tau^{\Lambda,\ubar{x},\bar{x}}}(x,r_i)-r_iw_{\tau^{\Lambda,\ubar{x},\bar{x}}}(x,r_i)=\lambda(x)\left(w_{\tau^{\Lambda,\ubar{x},\bar{x}}}(x,r_i)-K\right), \enskip \mbox{ for } i=1,2,
\end{align*}
which is seen by assuming for a moment that $F(r)$ is concentrated on $\{r_i\}$ and then applying Theorem \ref{theorem_ansatz} separately for $i=1,2$. 
Hence, using also \eqref{w-as-linear} and $\mathbf{A} = \frac{1}{2}\sigma^2x^2\frac{d^2}{dx^2}$, we find
\begin{align*}
x-r_i(a_ix+b_i)=\lambda(x)(a_ix+b_i-K)
\end{align*}
which gives us the candidate equilibrium stopping intensity function
\begin{align} \label{the-intensity-cand}
\lambda(x)& =\frac{(r_ia_i-1)x+r_ib_i}{K-a_ix-b_i}\nonumber\\
& = \frac{(pr_2 + (1-p)r_1)x-r_1r_2K}{K\left(pr_1 + (1-p)r_2\right)-x}\nonumber\\
& = \frac{(pr_2 + (1-p)r_1)x-r_1r_2K}{K\int_0^\infty rdF(r)-x}.
\end{align}
(Note that $i=1,2$ give the same function $\lambda(\cdot)$).  Our ansatz is now to determine the candidate for $\bar{x}$ as the point where $\lambda(\cdot)$ blows up, i.e., 
\begin{align}
\bar{x} \label{the-intensity-x-upper}
=K\int_0^\infty rdF(r).
\end{align}
We now want to identify $x\mapsto w_{\tau^{\Lambda,\ubar{x},\bar{x}}}(x,r_i),i=1,2$ on $[0,\ubar{x}]$. Since $X$ is a GBM we find, using standard theory, that $x\mapsto w_{\tau^{\Lambda,\ubar{x},\bar{x}}}(x,r_i),i=1,2$ (which are for the present specification of our problem defined in \eqref{w-def-realoption}) satisfy the differential equations
\begin{align*}
x+\frac{1}{2}\sigma^2x^2y_i''(x)-r_iy_i(x)=0
\end{align*}
with boundary conditions
\begin{align*}
y_i(0)&=0,\\
y_i(\ubar{x})&=a_i\ubar{x}+b_i.  
\end{align*}
The first boundary condition follows from basic properties of the GBM while the second boundary condition relies on continuity of $x\mapsto w_{\tau^{\Lambda,\ubar{x},\bar{x}}}(x,r_i)$ and \eqref{w-as-linear}. The solutions are given by
\begin{align*}
y_i(x)=\frac{x}{r_i}+D_ix^{\alpha(r_i)}
\end{align*}
where
\begin{align*}
\alpha(r_i):=\frac{1}{2}\left(1+\sqrt{\frac{8r_i}{\sigma^2}+1}\right),\quad D_i:=\frac{a_i\ubar{x}+b_i-\frac{\ubar{x}}{r_i}}{\ubar{x}^{\alpha(r_i)}}.
\end{align*}
Now, in order for our findings to be able to correspond to a mixed threshold strategy we must have that $\lambda(\ubar{x})\geq0$ which means that the numerator in \eqref{the-intensity-cand} should be nonnegative for $x=\ubar{x}$ and hence we need 
\begin{align*}
\ubar{x} & \geq 
\frac{r_1r_2K}{pr_2+(1-p)r_1}. 
\end{align*}
We now determine a candidate for $\ubar{x}$ using the smooth fit condition \eqref{mainthmcond4}, i.e.,
\begin{align*}
J'_{\tau^{\Lambda,\ubar{x},\bar{x}}}(\ubar{x})=0,
\end{align*}
which is equivalent to 
$pw'_{\tau^{\Lambda,\ubar{x},\bar{x}}}(\ubar{x},r_1) + (1-p)w'_{\tau^{\Lambda,\ubar{x},\bar{x}}}(\ubar{x},r_2)
= py_1'(\ubar{x}) + (1-p)y_2' (\ubar{x}) =0$; from which we obtain the candidate
\begin{align}\label{u_x_eq}
\ubar{x}
& 
=\frac{p\alpha(r_1)b_1+(1-p)\alpha(r_2)b_2}{\int_0^\infty\frac{\alpha(r)-1}{r}dF(r)-\left(p\alpha(r_1)a_1+(1-p)\alpha(r_2)a_2\right)}\nonumber\\
& =\frac{ (\alpha(r_2)r_1-\alpha(r_1)r_2)K       }{(r_1-r_2) \int_0^\infty\frac{\alpha(r)-1}{r}dF(r)-\left(\alpha(r_1)-\alpha(r_2) \right)}.
\end{align}
%
%
%
%
Note that our ansatz can only correspond to a mixed threshold equilibrium  if $0 <\ubar{x}<\bar{x}$ and Lemma \ref{help-lemma-GBM} (below) thus implies that we have a candidate for a mixed equilibrium only if the model parameters are such that \eqref{mixed_eq_cond} holds.   
\begin{lemma} \label{help-lemma-GBM} 
The candidate value for $\ubar{x}$ in \eqref{u_x_eq} satisfies 
$\ubar{x}\in\left[\frac{r_1r_2K}{pr_2+(1-p)r_1},\bar{x}\right)$ if and only if
\begin{align}
\int_0^\infty\alpha(r)dF(r)<\int_0^\infty rdF(r)\int_0^\infty\frac{\alpha(r)-1}{r}dF(r).\label{mixed_eq_cond}
\end{align}
\end{lemma}
\begin{proof} This follows from the definitions of $\ubar{x}$ and $\bar{x}$ and calculations.  
\end{proof}

Finally, using Theorem \ref{local_ansatz} we determine a candidate for the local time push  according to 
\begin{align} \label{the-candidate-lambda}
\lambda 
& =\frac{w'_{\tau^{\Lambda,\ubar{x},\bar{x}}}(\ubar{x}+,r_i)-w'_{\tau^{\Lambda,\ubar{x},\bar{x}}}(\ubar{x}-,r_i)}{2\left(w_{\tau^{\Lambda,\ubar{x},\bar{x}}}(\ubar{x},r_i)-K\right)}\nonumber\\
& =\frac{a_i-\frac{\alpha(r_i)}{\ubar{x}}\left(a_i\ubar{x}+b_i-\frac{\ubar{x}}{r_i}\right)-\frac{1}{r_i}}{2(a_i\ubar{x}+b_i-K)}.
\end{align}
To see that both $i=1,2$ give the same $\lambda$ use 
\begin{align*}
pw'_{\tau^{\Lambda,\ubar{x},\bar{x}}}(\ubar{x}-,r_1)=-(1-p)w'_{\tau^{\Lambda,\ubar{x},\bar{x}}}(\ubar{x}-,r_2)
\end{align*}
(which holds by the smooth fit condition $J'_{\tau^{\Lambda,\ubar{x},\bar{x}}}(\ubar{x})=0$), 
\begin{align*}
p(w_{\tau^{\Lambda,\ubar{x},\bar{x}}}(\ubar{x},r_1)-K) = - (1-p)(w_{\tau^{\Lambda,\ubar{x},\bar{x}}}(\ubar{x},r_2)-K)
\end{align*}
(which holds by $J_{\tau^{\Lambda,\ubar{x},\bar{x}}}(\ubar{x})=K$), and 
\begin{align*}
pw'_{\tau^{\Lambda,\ubar{x},\bar{x}}}(\ubar{x}+,r_1) = -(1-p)pw'_{\tau^{\Lambda,\ubar{x},\bar{x}}}(\ubar{x}+,r_2).
\end{align*}
In order for $\lambda$ in \eqref{the-candidate-lambda} to correspond to a mixed threshold strategy we need that  $\lambda>0$, and we conclude with the following result. 

\begin{proposition} \label{candidate-is-really-a-strat} Suppose \eqref{mixed_eq_cond} holds. Then the candidate threshold mixed equilibrium $\tau^{\Lambda,\ubar{x},\bar{x}}$ corresponding to 
\eqref{the-intensity-cand}, \eqref{the-intensity-x-upper}, \eqref{u_x_eq} and \eqref{the-candidate-lambda} satisfies 
$0 <\ubar{x}<\bar{x}$, 
$\lambda(\cdot)\geq 0$ on $[\ubar{x},\bar{x})$, 
and $\lambda>0$ (and it is therefore a proper mixed threshold strategy, cf. Definition \ref{def:mixed_threshold_stopping_time}). 
\end{proposition} 
\begin{proof}
The claims follow from Lemma \ref{help-lemma-GBM} and calculations. 
\end{proof}

%

\subsection{Verification, equilibrium existence and a smooth fit principle}\label{sec:verification}
Let us first summarize the candidate mixed threshold equilibrium $\tau^{\Lambda,\ubar{x},\bar{x}}$. The candidate interval $[\ubar{x},\bar{x})$ on which randomized stopping occurs is given by 
\begin{align*} 
\ubar{x}=\frac{ (\alpha(r_2)r_1-\alpha(r_1)r_2)K       }{(r_1-r_2) \int_0^\infty\frac{\alpha(r)-1}{r}dF(r)-\left(\alpha(r_1)-\alpha(r_2) \right)}
\end{align*}
and
\begin{align*}
\bar{x} = K\int_0^\infty rdF(r).
\end{align*}
The candidate stopping intensity on $[\ubar{x},\bar{x})$ is given by 
\begin{align*}
\lambda(x)=\frac{(pr_2 + (1-p)r_1)x-r_1r_2K}{K\int_0^\infty rdF(r)-x}
\end{align*}
and the candidate for local time push at $\ubar{x}$ is given by
\begin{align*}
\lambda=\frac{a_1-\frac{\alpha(r_1)}{\ubar{x}}\left(a_1\ubar{x}+b_1-\frac{\ubar{x}}{r_1}\right)-\frac{1}{r_1}}{2(a_1\ubar{x}+b_1-K)}.
\end{align*}

\begin{theorem}\label{example_eq_theorem}

If \eqref{mixed_eq_cond} holds, then the candidate mixed threshold equilibrium is indeed an equilibrium.

\end{theorem}
\begin{proof}
We first note that the candidate equilibrium is a mixed threshold strategy by Proposition \ref{candidate-is-really-a-strat}. Section \ref{sec:candidate} indicates that $w_{\tau^{\Lambda,\ubar{x},\bar{x}}}(x,r_i)$ defined according to \eqref{w-def-realoption} should satisfy $w_{\tau^{\Lambda,\ubar{x},\bar{x}}}(x,r_i)=h_i(x)$ given that
\begin{align*}
h_i(x):=\begin{cases}D_ix^{\alpha(r_i)}+\frac{x}{r_i}& x\in[0,\ubar{x}]\\
a_ix+b_i&x\in [\ubar{x},\bar{x}]\\
K&x\in[\bar{x},\infty).
\end{cases} 
	\end{align*}
We will first show that this is indeed so. It is directly seen that this holds for $x=0$ and $x\in [\bar{x},\infty)$. Now note (cf. calculations in Section \ref{sec:candidate}) that $h_i\in \mathcal{C}^2((0,\ubar{x})\cup(\ubar{x},\bar{x}))\cap \mathcal{C}(0,\infty)$ is the solution to the boundary value problem 
\begin{align*}
x + (\mathbf{A}-r_i)h_i(x)&=\lambda(x)(h_i(x)-K), \enskip x \in (0,\ubar{x})\cup(\ubar{x},\bar{x}) \\
h_i(0)&=0\\
h_i(x)&=K,\enskip x \in [\bar{x},\infty)\\
h_i'(\ubar{x}-) & = D_i{\alpha(r_i)}x^{\alpha(r_i)-1}+\frac{1}{r_i}\\
h_i'(\ubar{x}+) & = a_i
\end{align*}
where $\lambda(x):=0$ on $(0,\ubar{x})$. Calculations also give that
\begin{align*}
h_i'(\ubar{x}+) - h_i'(\ubar{x}-) = 2\lambda\left(h_i(\ubar{x})-K\right).
\end{align*}
Using first the local time-space formula and second the findings above we obtain for $x\in (0,\bar{x})$ that
\begin{align} \label{jajajajui1}
& \E_x\left(e^{-r_i\tau^{\Lambda,\ubar{x},\bar{x}}}h_i\left(X_{\tau^{\Lambda,\ubar{x},\bar{x}}}\right)\right)-h_i(x) \nonumber \\
& \enskip =\E_x\left(\int_0^{\tau^{\Lambda,\ubar{x},\bar{x}}}e^{-r_it}(\mathbf{A}-r)h_i(X_t)I_{\{X_t\neq \ubar{x}\}}dt\right) +\frac{1}{2}\E_x\left(\int_0^{\tau^{\Lambda,\ubar{x},\bar{x}}}e^{-r_it}(h_i'(\ubar{x}+)-h_i'(\ubar{x}-))dl_{t}^{\ubar{x}}\right) \nonumber \\
& \enskip =-\E_x\left(\int_0^{\tau^{\Lambda,\ubar{x},\bar{x}}}e^{-r_it}X_tdt\right)+ 
\E_x\left(\int_0^{\tau^{\Lambda,\ubar{x},\bar{x}}}e^{-r_it}(h_i(X_t)-K)(\lambda(X_t)dt+\lambda dl_{t}^{\ubar{x}})\right).
\end{align}
In the rest of the proof we let $D:=(0,\bar{x})$ and we will use the observation  
\begin{align}\label{tau-lambda-before-tau-D}
I_{\{\tau^{\Lambda}\leq \tau^D\}}= 1 \enskip \mbox{a.s.,}
\end{align}
which can be seen using \cite[Theorem 2.6, Equation (2.12)]{mijatovic2012convergence} 
%
%
which implies that 
$\int_0^{\tau^D+h}\lambda(X_s)ds=\infty$, for any $h>0$, which implies that 
$\inf\{t\geq0: \int_0^t\lambda(X_s)ds \geq U\} \leq \tau^D$, which implies \eqref{tau-lambda-before-tau-D}.
Now, it follows from \eqref{tau-lambda-before-tau-D}  that $\tau^{\Lambda,\ubar{x},\bar{x}}:=\tau^{\Lambda} \wedge \tau^D=\tau^{\Lambda}$ a.s. and using also Proposition \ref{centralProposition} (as well as, e.g., Definition \ref{def:mixed_threshold_stopping_time}) we thus find
\begin{align}  \label{jajajajui2}
\E_x\left(\int_0^{\tau^{\Lambda,\ubar{x},\bar{x}}}e^{-r_it}(h_i(X_t)-K)(\lambda(X_t)dt+\lambda dl_{t}^{\ubar{x}})\right)
= 
\E_x\left(e^{-r_i\tau^{\Lambda,\ubar{x},\bar{x}}}(h_i(X_{\tau^{\Lambda,\ubar{x},\bar{x}}})-K)\right). 
\end{align}
By rearranging \eqref{jajajajui1} and \eqref {jajajajui2} and recalling \eqref{w-def-realoption} we conclude that $w_{\tau^{\Lambda,\ubar{x},\bar{x}}}(x,r_i)=h_i(x)$. Using this result together with \eqref{disc-F-realOp}, observations from Section \ref{sec:candidate} and calculations, gives us 
\begin{align*}
J_{\tau^{\Lambda,\ubar{x},\bar{x}}}(x)& =ph_1(x)+(1-p)h_2(x)\\
& =
\begin{cases}
p \left(D_1x^{\alpha(r_1)}+\frac{x}{r_1}\right) + (1-p)\left(D_2x^{\alpha(r_2)}+\frac{x}{r_2} \right)           & x\in[0,\ubar{x}]\\
p(a_1x+b_1) +(1-p)(a_2x+b_2)                                   &x\in [\ubar{x},\bar{x}]\\
K                                          &x\in[\bar{x},\infty)
\end{cases}\\
& =
\begin{cases}
p \left(
\left(a_1\ubar{x}+b_1-\frac{\ubar{x}}{r_1}\right) \left(\frac{x}{\ubar{x}}\right)^{\alpha(r_1)}
+\frac{x}{r_1}\right)
+ (1-p)\left(
\left(a_2\ubar{x}+b_2-\frac{\ubar{x}}{r_2}\right) \left(\frac{x}{\ubar{x}}\right)^{\alpha(r_2)}
+\frac{x}{r_2} \right)           & x\in[0,\ubar{x}]\\
K                                   &x\in [\ubar{x},\bar{x}]\\
K                                          &x\in[\bar{x},\infty).
\end{cases}
\end{align*}
It is now directly seen that condition \eqref{mainthmcond2} (in Theorem \ref{main_thm}) is satisfied whereas \eqref{mainthmcond3} and \eqref{mainthmcond4} can be verified. Moreover, it can be found that $J_{\tau^{\Lambda,\ubar{x},\bar{x}}}(x)$ is a nonlinear concave function on $(0,\ubar{x})$ and from this together with $J'_{\tau^{\Lambda,\ubar{x},\bar{x}}}(\ubar{x})=0$ it follows that \eqref{mainthmcond1} holds. It follows from Theorem \ref{main_thm} that we have indeed found an equilibrium. 
\end{proof}


A main finding of \cite{tan2021failure} is that in case \eqref{mixed_eq_cond} does not hold then a pure threshold stopping time equilibrium exists. Theorem \ref{pure_eq_theorem} below is an analogous result for the setting of the present paper. Note, however, that in our case also mixed strategies are allowed as deviation strategies which gives a stronger result.


\begin{theorem} \label{pure_eq_theorem} 
If \eqref{mixed_eq_cond} does not hold, then $\tau^{\Lambda,\ubar{x},\bar{x}}:=\inf\{t \geq 0: X_t\geq \ubar{x}\}$ with 
\begin{align} \label{pureNE-thres}
\ubar{x}:= \frac{\int_0^\infty\alpha(r)dF(r)}{\int_0^\infty\frac{\alpha(r)-1}{r}dF(r)}K,
\end{align}
is  a pure threshold equilibrium. 
\end{theorem}
\begin{proof}
Using similar arguments as in Theorem \ref{example_eq_theorem}, we obtain $w_{\tau^{\Lambda,\ubar{x},\bar{x}}}(x,r_i)=h_i(x)$, where
\begin{align*}
h_i(x):=\begin{cases}\left(K-\frac{\ubar{x}}{r_i}\right)\left(\frac{x}{\ubar{x}}\right)^{\alpha(r_i)}+\frac{x}{r_i}& x\in[0,\ubar{x}]\\
K&x\in[\bar{x},\infty)
\end{cases}
\end{align*}
which is the $\mathcal{C}^2(0,\ubar{x})\cap \mathcal{C}(0,\infty)$ solution to the boundary value problem
\begin{align*}
x +(\mathbf{A}-r_i)h_i(x)&=0, \enskip x \in [0,\ubar{x}]\\
h_i(0)&=0\\
h_i(x)&=K,\enskip \enskip x \in [\ubar{x},\infty).
\end{align*}
It can now be verified using \eqref{pureNE-thres} that
 \begin{align*}
J'_{\tau^{\Lambda,\ubar{x},\bar{x}}}(\ubar{x}-)
& = p h_1'(\ubar{x}-) + (1-p)h_2'(\ubar{x}-)\\
& = 0,
\end{align*}
which implies that the smooth fit condition  \eqref{mainthmcond4} holds. It can be verified that \eqref{mainthmcond3} holds if and only if \eqref{mixed_eq_cond} does not hold; whereas  \eqref{mainthmcond2} is, since we have a pure threshold strategy, void. Condition \eqref{mainthmcond1} is verified using $J'_{\tau^{\Lambda,\ubar{x},\bar{x}}}(\ubar{x})=0$ and by verifying that $J_{\tau^{\Lambda,\ubar{x},\bar{x}}}(x)=ph_1(x)+(1-p)h_2(x)$ is a concave (and nonlinear) function on $(0,\ubar{x})$. The result follows from Theorem \ref{main_thm}. 
\end{proof}

Illustrations are found in Figure \ref{fig1}; note that ${\tau^{\Lambda,\ubar{x},\bar{x}}}$ has for ease of  exposition been notationally suppressed in the cost functions in the figure legend.

\begin{figure}[h]\label{fig1}
  \subfigure[]{\includegraphics[width=0.45\textwidth]{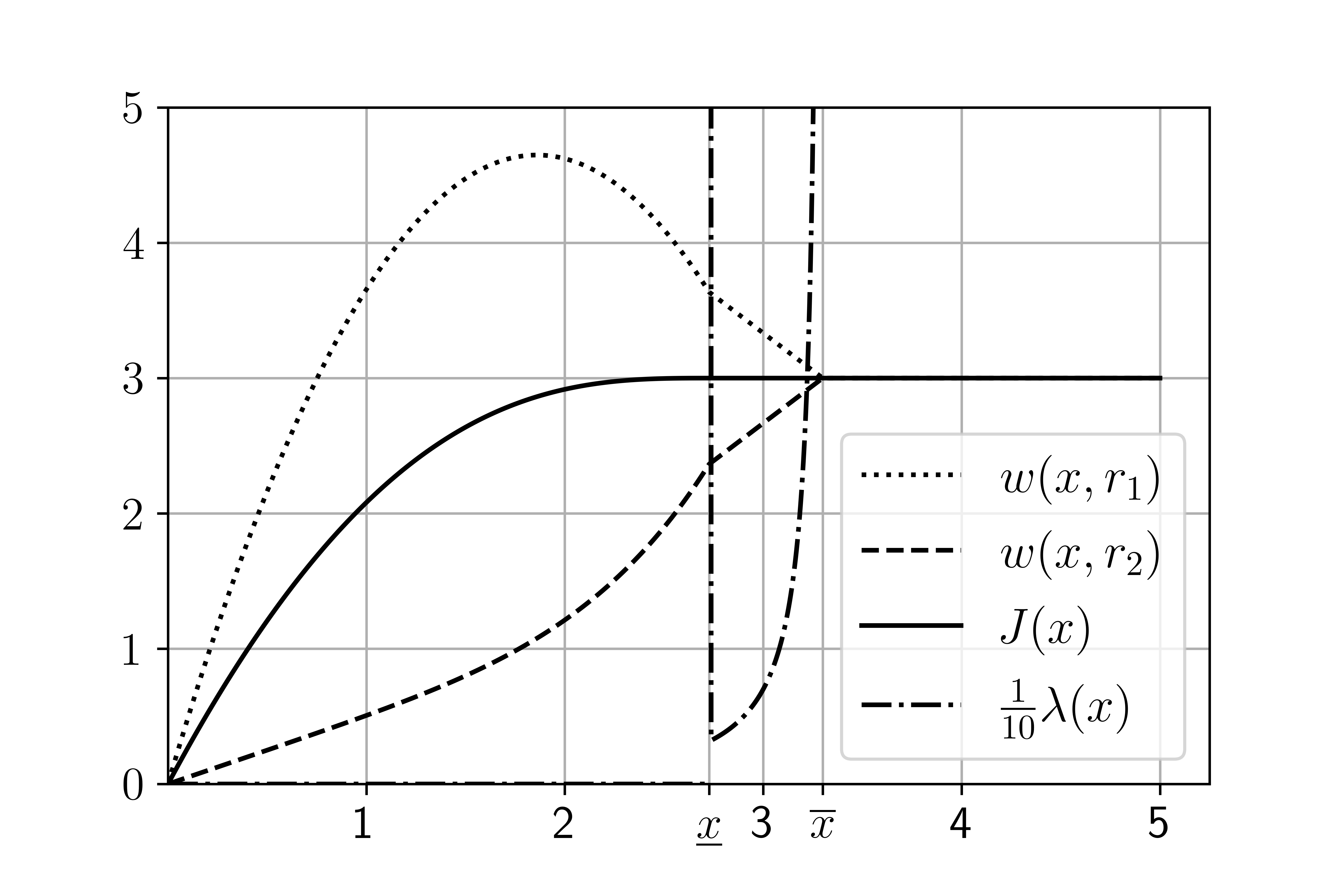}\label{fig:mixed1}}
  \subfigure[]{\includegraphics[width=0.45\textwidth]{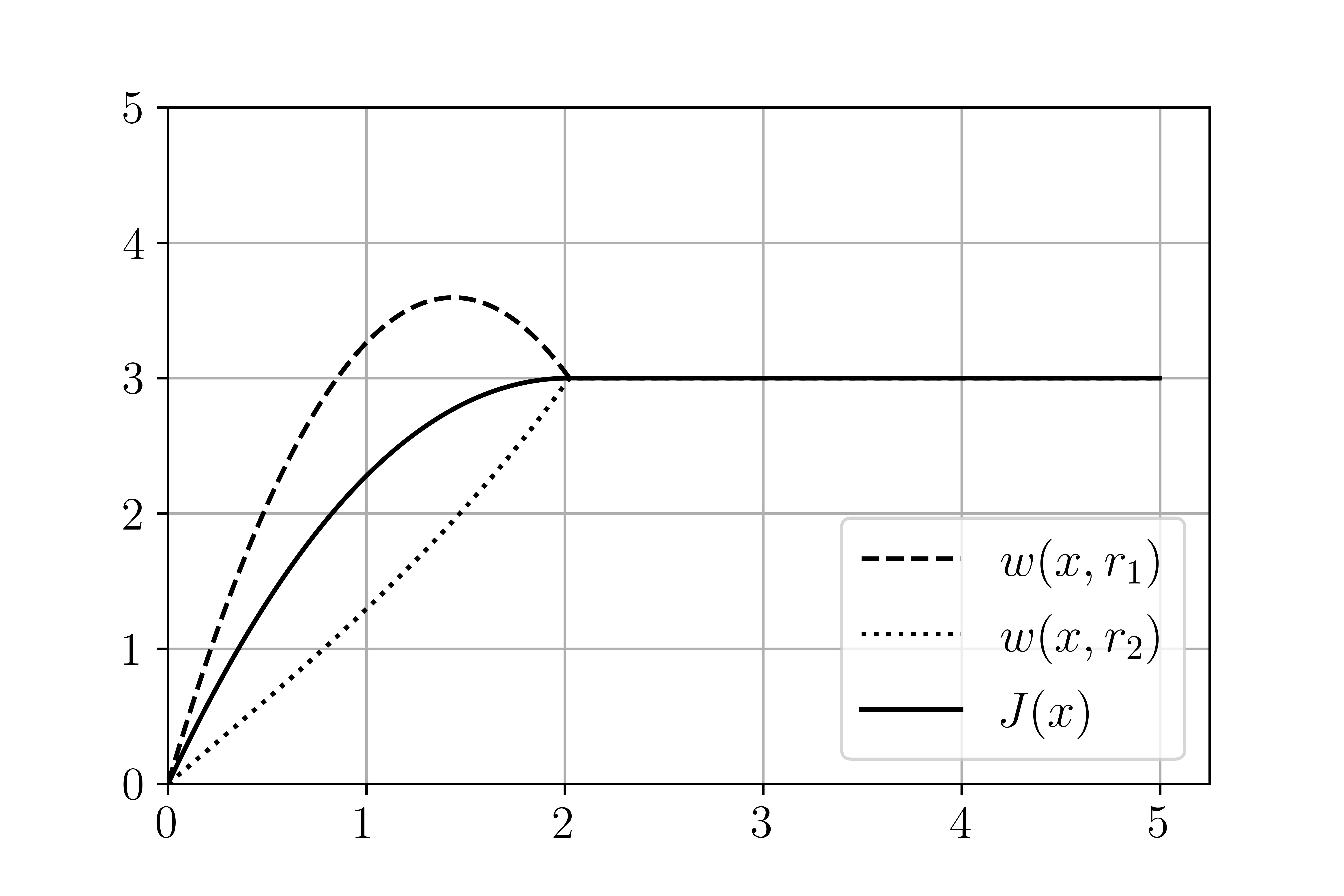}\label{fig:pure1}}
\caption{
(a): The mixed threshold equilibrium cost functions in the case $\sigma^2=0.2$, $K=3$, $p=0.5$, $r_1=0.2$, $r_2=2$. The singular local time push at $\protect\underaccent{\bar}{x}$ is motivated by the peak in the stopping intensity $\lambda(x)$. 
(b): The pure equilibrium cost functions in the case $\sigma^2=0.2$, $K=3$, $p=0.5$, $r_1=0.2$, $r_2=0.8$.}
\end{figure}

\begin{corollary}\label{smooth_fit_corrolary} 
A threshold equilibrium ${\tau^{\Lambda,\ubar{x},\bar{x}}}$ of either pure or mixed type always exists. In both cases smooth fit is satisfied in the sense that
\begin{align*}
J'_{\tau^{\Lambda,\ubar{x},\bar{x}}}(\ubar{x}) = 0.  
\end{align*}
 \end{corollary}

\begin{proof} See the two previous result and their proofs.
\end{proof}

\begin{remark} It can be shown that condition \eqref{mixed_eq_cond} holds if $r_2-r_1$ is sufficiently large (see p. 11 in \cite{tan2021failure} for an argument). 
%
%
\end{remark}





%

\appendix
\section{Appendix}

\subsection{Results related to local time and the exit time $\tau_h$}

\begin{lemma} \label{lemma-tau-h-oct}  For any $x \in (\alpha,\beta)$,
\begin{align*}
\lim_{h\searrow 0} \frac{h^2}{\E_x(\tau_h)} = \sigma^2(x),
\enskip \lim_{h\searrow 0} \frac{\E_x\left(\tau_h^2\right)}{\E_x(\tau_h)} = 0,
\enskip \mbox{ and } \enskip 
\enskip \lim_{h\searrow 0} \frac{\left(  \E_x\left(l^x_{\tau_h}\right)  \right)^2}{\E_x(\tau_h)} = \sigma^2(x).
\end{align*}
\end{lemma}
\begin{proof} 
See \cite[Lemma A.5 and Proposition 3.3]{christensen2020time}. \end{proof}

%

\begin{lemma}\label{prop_limit_local_time_time_tau-h}  For any $x \in (\alpha,\beta)$, 
\begin{align*}
\lim_{h\searrow 0} \frac{\E_x\left(\tau_hl_{\tau_h}^x\right)}{ \E_x(\tau_h)}=0 \enskip \mbox{ and } \enskip 
\lim_{h\searrow 0} \frac{\E_x\left((l_{\tau_h}^x)^2\right)}{ \E_x(\tau_h)}=2\sigma^2(x).
\end{align*}
For any $x \in (\alpha,\beta)$ and any process $\Psi$ (cf. Definition \ref{def:mix_strat}), 
\begin{align*}
\lim_{h\searrow 0}\frac{\E_x\left(l_{\tau^\Psi\wedge\tau_h}^x\right)}{\sqrt{\E_x(\tau_h)}}=\sigma(x)
\enskip \mbox{ and } \enskip 
\lim_{h\searrow 0}\frac{\E_x\left(\tau^{\Psi}  \wedge \tau_h\right)}{\E_x(\tau_h)}=1.
\end{align*}
\end{lemma}

\begin{proof}
Using the Cauchy-Schwartz inequality we obtain
\begin{align*}
\frac{\E_x\left(\tau_hl_{\tau_h}^x\right)}{ \E_x(\tau_h)}=\E_x\left(\frac{\tau_h}{\sqrt{ \E_x(\tau_h)}}\cdot\frac{l_{\tau_h}^x}{\sqrt{ \E_x(\tau_h)}}\right)
\leq \sqrt{\frac{\E_x\left(\tau_h^2\right)}{ \E_x(\tau_h)}\frac{\E_x\left((l_{\tau_h}^x)^2\right)}{ \E_x(\tau_h)}}.
\end{align*}
Hence the first statement follows from Lemma \ref{lemma-tau-h-oct} and the second statement, which we now prove. Using the definition of the local time we obtain 

\begin{align*}
\E_x\left((l_{\tau_h}^x)^2\right)=h^2-2h\E_x\left(\int_0^{\tau_h}sgn(X_s-x)dX_s\right)+\E_x\left(\left(\int_0^{\tau_h}sgn(X_s-x)dX_s\right)^2\right).
\end{align*}
For the first two terms we obtain using Lemma \ref{lemma-tau-h-oct} 
\begin{align*}
\frac{h^2-2h\E_x\left(\int_0^{\tau_h}sgn(X_s-x)\mu(X_s)ds\right)}{\E_x(\tau_h)}\rightarrow \sigma^2(x)
\end{align*}
as $h\searrow 0$.
For the last term we obtain using the It\^{o}-isometry
\begin{align*}
\E_x\left(\left(\int_0^{\tau_h}sgn(X_s-x)dX_s\right)^2\right)
&=\E_x\left(\left(\int_0^{\tau_h}sgn(X_s-x)\mu(X_s)ds\right)^2\right)\\
&\enskip\enskip
+2\E_x\left(\int_0^{\tau_h}sgn(X_s-x)\mu(X_s)ds\int_0^{\tau_h}sgn(X_s-x)\sigma(X_s)dW_s \right)\\&
\enskip\enskip+\E_x\left(\int_0^{\tau_h}\sigma^2(X_s)ds\right).
\end{align*}
We have that
\begin{align*}
\lim_{h\searrow 0}\left|\frac{\E_x\left(\left(\int_0^{\tau_h}sgn(X_s-x)\mu(X_s)ds\right)^2\right)}{\E_x(\tau_h)}\right|\leq \lim_{h\searrow 0}\sup_{y\in [x-h,x+h]}\mu^2(y)\frac{\E_x\left(\tau_h^2\right)}{\E_x(\tau_h)}=0
\end{align*}
and using Cauchy-Schwartz and It\^{o}-isometry we obtain
\begin{align*}
&\lim_{h\searrow 0}\left|\frac{\E_x\left(\int_0^{\tau_h}sgn(X_s-x)\mu(X_s)ds\right)\left(\int_0^{\tau_h}sgn(X_s-x)\sigma(X_s)dW_s\right)}{\E_x(\tau_h)}\right|\\
&\enskip \leq \lim_{h\searrow 0}
\frac{\sqrt{
\E_x\left(\left(\int_0^{\tau_h}sgn(X_s-x)\mu(X_s)ds\right)^2\right)
\E_x\left(\left(\int_0^{\tau_h}sgn(X_s-x)\sigma(X_s)dW_s\right)^2\right)
}}{\E_x(\tau_h)}
\\
&\enskip \leq \lim_{h\searrow 0}\sup_{y_1,y_2\in[x-h,x+h]}|\mu(y_1)\sigma(y_2)| \sqrt{\frac{\E_x\left(\tau_h^2\right)}{\E_x\left(\tau_h\right)}\frac{\E_x\left(\tau_h\right)}{\E_x\left(\tau_h\right)}}=0.
\end{align*}
Finally,
\begin{align*}
\lim_{h\searrow 0}\frac{\E_x\left(\int_0^{\tau_h}\sigma(X_s)^2ds\right)}{\E_x\left(\tau_h\right)}=\sigma^2(x),
\end{align*}
which finishes the proof of the second result. Let us now prove the third result. Using, e.g., Cauchy-Schwartz and that the local time is an increasing non-negative process it is directly seen that 
\begin{align*}
1 
& \geq \frac{\E_x\left(l_{\tau^\Psi\wedge\tau_h}^x\right)}{\E_x\left(l_{\tau_h}^x\right)}\\
& \geq \frac{\E_x\left(l_{\tau_h}^x  - l_{\tau_h}^x I_{\{ \tau^\Psi< \tau_h \}}\right)}{\E_x\left(l_{\tau_h}^x\right)}\\
& = 1 -  \frac{\E_x\left( l_{\tau_h}^x I_{\{ \tau^\Psi< \tau_h \}}\right)}{\E_x\left(l_{\tau_h}^x\right)}\\
&\geq  1 -  \sqrt{\frac{\E_x\left( \left(l_{\tau_h}^x\right)^2\right)}{\left(\E_x\left(l_{\tau_h}^x\right)\right)^2}\mathbb{P}_x\left(\tau^\Psi< \tau_h\right).}
\end{align*}
Sending $h\searrow 0$ implies that the ratio of expected values in the square root in the last term converges to a constant---to see this use the calculations above and Lemma \ref{lemma-tau-h-oct}---while the probability converges to $0$. This implies that 
\begin{align*} 
\lim_{h\searrow 0} \frac{\E_x\left(l_{\tau^\Psi\wedge\tau_h}^x\right)}{\E_x\left(l_{\tau_h}^x\right)} = 1,
\end{align*} 
which together with the last result of Lemma \ref{lemma-tau-h-oct} implies that the third result holds. The last result can be proved using the same techniques as for the third result---relying also on the fact that
$\frac{\E_x\left(\tau_h^2\right)}{\left(\E_x(\tau_h)\right)^2}$ converges to a constant as $h \searrow 0$, which can be seen using the same argument as in the proof of Lemma \cite[Lemma A.5]{christensen2020time}.

\end{proof}

\subsection{A result used in Section \ref{sec:main} and the proof of Lemma \ref{main_lemma}}

\begin{lemma}\label{mp_lemma} For any $\tau^{\Lambda,\ubar{x},\bar{x}} \in \mathcal{N}$, it holds that 
\begin{align*} 
J_{\tau^{\Lambda,\ubar{x},\bar{x}}\circ \theta_{\tau_h}+\tau_h}(x)=\int_0^\infty\E_x\left(  \int_0^{\tau_h}e^{-rs}f(X_s)ds+e^{-r\tau_h}w_{\tau^{\Lambda,\ubar{x},\bar{x}}}(X_{\tau_h},r)\right)dF(r).
\end{align*}
\end{lemma}
\begin{proof}
By Fubini and the definition of the WDF we obtain
\begin{align*}
J&_{\tau^{\Lambda,\ubar{x},\bar{x}}\circ \theta_{\tau_h}+\tau_h}(x)-\int_0^\infty\E_x\left(  \int_0^{\tau_h}e^{-rs}f(X_s)ds\right)dF(r)\\
&=\int_0^\infty\E_x\left(e^{-r\tau_h}\left(\int_{0}^{\tau^{\Lambda,\ubar{x},\bar{x}}\circ \theta_{\tau_h}}e^{-rs}f(X_{s+\tau_h})ds+e^{-r\tau^{\Lambda,\ubar{x},\bar{x}}\circ \theta_{\tau_h}}g\left(X_{\tau^{\Lambda,\ubar{x},\bar{x}}\circ \theta_{\tau_h}+\tau_h}\right)\right)\right)dF(r).
\end{align*}
Here, the term inside the expectation represents the exponentially discounted cost of a process starting at $X_{\tau_h}$ and using the strong Markov property we obtain
\begin{align*}
J&_{\tau^{\Lambda,\ubar{x},\bar{x}}\circ \theta_{\tau_h}+\tau_h}(x)-\int_0^\infty\E_x\left(  \int_0^{\tau_h}e^{-rs}f(X_s)ds\right)dF(r)
=\int_0^\infty \E_x\left(e^{-r\tau_h}w_{\tau^{\Lambda,\ubar{x},\bar{x}}}(X_{\tau_h},r)\right)dF(r).
\end{align*}
Rearranging completes the proof.
\end{proof}

\begin{proof}[Proof of Lemma \ref{main_lemma}] 
%

Using arguments similar to those of the proof of Lemma \ref{main_lemma_local} we find, that for any sufficiently small $h>0$, it holds that
\begin{align*}
&J_{\tau^{\Lambda,\ubar{x},\bar{x}}\circ \theta_{\tau_h}+\tau_h}(x)-J_{\tau^{\Lambda,\ubar{x},\bar{x}} \diamond \tau^{\Psi,\mathcal{X},D}(h)}(x)\\
\enskip &=\int_0^\infty\E_x\left( I_{\{\tau^\Psi \leq \tau_h\}}\left(e^{-r\tau_h}w_{\tau^{\Lambda,\ubar{x},\bar{x}}}(X_{\tau_h},r)+\int_{\tau_\Psi}^{\tau_h} e^{-rs} f(X_s)ds-e^{-r\tau^\Psi}g(X_{\tau^\Psi})\right)\right)dF(r)\\
\enskip &=\int_0^\infty\E_x\left(\int_0^{\tau_h}
\left(e^{-r\tau_h}w_{\tau^{\Lambda,\ubar{x},\bar{x}}}(X_{\tau_h},r)+\int_{t}^{\tau_h} e^{-rs}f(X_s)ds-e^{-rt}g(X_t)
\right)\psi(X_t)e^{-\int_0^t\psi(X_s)ds}dt\right)dF(r)
\end{align*}
where we used that $d\mathbb{P}_x\left( \tau^\Psi  \leq t \enskip \vline \enskip {\cal F} ^X_\infty  \right)=e^{-\Psi_t}d\Psi_t = \psi(X_t)e^{-\int_0^t\psi(X_s)ds}dt$ on the stochastic interval $[0,\tau_h]$ for sufficiently small $h>0$ (since $ x\notin \cal X$). 
Now use the continuity of the paths of $X$ in combination with the fact that $\psi(X_t)$ is bounded on $[0,\tau_h]$, to obtain the desired limit of the expectation as follows. For the middle term we find 
\begin{align*}
\left|\frac{\int_0^\infty\E_x\left(\int_0^{\tau_h}\psi(X_t)e^{-\int_0^t\psi(X_s)ds}\int_{t}^{\tau_h} e^{-rs}f(X_s)dsdt\right)dF(r)}{\E_x(\tau_h)}\right|
\leq \sup_{y_1,y_2\in [x-h,x+h]}\psi(y_1)f(y_2)\frac{\E_x\left(   \tau_h^2  \right)}{\E_x(\tau_h)} 
\rightarrow 0,
\end{align*}
as $h \searrow 0$; where we have used Lemma \ref{lemma-tau-h-oct}. For the last term we obtain
\begin{align*}
&\frac{\int_0^\infty \E_x\left(
\int_0^{\tau_h}\psi(X_t)e^{-\int_0^t\psi(X_s)ds}e^{-rt}g(X_t)dt\right)dF(r)}{\E_x(\tau_h)}\\ \enskip 
&=\frac{ \E_x\left(
\int_0^{\tau_h}\psi(X_t)g(X_t)dt\right)}{\E_x(\tau_h)}-
\frac{\int_0^\infty \E_x\left(
\int_0^{\tau_h}\psi(X_t)\left(1-e^{-\left(\int_0^t\psi(X_s)ds +rt\right)}\right)g(X_t)dt\right)dF(r)}{\E_x(\tau_h)}.
\end{align*}
Taking a look at the terms separately, we obtain using that $\psi(\cdot)$ is RCLL and basic properties of diffusions that
\begin{align*}
\lim_{h\searrow 0}\frac{\E_x\left(\int_0^{\tau_h}\psi(X_t)g(X_t)dt\right)}{\E_x(\tau_h)}=\frac{1}{2}(\psi(x-)+\psi(x))g(x),
\end{align*}
and using $|1-e^{-y}|\leq y$ for $y\geq 0$ we find
\begin{align*}
&\left|\frac{\int_0^\infty \E_x\left(
\int_0^{\tau_h}\psi(X_t)\left(1-e^{-\left(\int_0^t\psi(X_s)ds +rt\right)}\right)g(X_t)dt\right)dF(r)}{\E_x(\tau_h)}\right|\\
& \leq
\left|\frac{\int_0^\infty\E_x\left(
\int_0^{\tau_h}\psi(X_t)\left(\int_0^{\tau_h}\psi(X_s)ds +r \tau_h  \right)g(X_t)dt\right)dF(r)}{\E_x(\tau_h)}\right|\\
&\enskip \leq \sup_{y_1,y_2,y_3\in[x-h,x+h]}g(y_1)\psi(y_2)\left(\psi(y_3)+\int_0 ^\infty rdF(r)\right) \frac{\E_x\left(   \tau_h^2  \right)}{\E_x(\tau_h)} 
\rightarrow 0, \enskip \mbox{ as } \enskip h \searrow 0.
\end{align*}
Finally, we take a look at the first term, for which we, using similar arguments as above (and also \eqref{J-in-terms-w}) obtain 
\begin{align*}
\lim_{h\searrow 0}\frac{\int_0^\infty \E_x\left(\int_0^{\tau_h}\psi(X_t)e^{-\int_0^t\psi(X_s)ds}e^{-r\tau_h}w_{\tau^{\Lambda,\ubar{x},\bar{x}}}(X_{\tau_h},r)dt\right)dF(r)}{\E_x(\tau_h)}=\frac{1}{2}(\psi(x-)+\psi(x))J_{\tau^{\Lambda,\ubar{x},\bar{x}}}(x).
\end{align*}
The proof of the second result is analogous, which can be seen by setting $(\alpha,\bar{x})=D$ and using \eqref{sec_resul_lambda}. 
\end{proof}

\bibliographystyle{abbrv}
\bibliography{TimeInconStopGBMAndKriSor20220630}

\end{document}